\documentclass[11pt,
]{amsart}
\usepackage{
amssymb, graphicx, color, cite
}
\usepackage[T1]{fontenc}

\renewcommand{\Re}{\textrm{Re}\,} 

\date{\today}
\usepackage[linkcolor=blue, citecolor=black, breaklinks=true]{hyperref}

\usepackage{mathtools}

\mathtoolsset{showonlyrefs}
\renewcommand{\r}[1]{\eqref{#1}}
\numberwithin{equation}{section}%

\newtheorem{theorem}{Theorem}
\newtheorem{lemma}{Lemma}[section]
\newtheorem{proposition}[lemma]{Proposition}
\newtheorem{corollary}[lemma]{Corollary}

\theoremstyle{definition}

\DeclareMathOperator{\tr}{tr}

\DeclareMathOperator{\supp}{supp}

\DeclareMathOperator{\dist}{dist}
\newcommand{\eps}{\varepsilon}

\newcommand{\R}{{\mathbb R}}
\renewcommand{\S}{{\mathbb S}}
\newcommand{\C}{{\mathbb C}}
\newcommand{\Id}{\mbox{\rm Id}}

\newcommand{\be}[1]{\begin{equation}\label{#1}}
\newcommand{\ee}{\end{equation}}

\renewcommand{\d}{\mathrm{d}}

\renewcommand{\i}{\mathrm{i}}

\newcommand{\bo}{{\partial M}}
\newcommand{\Mint}{M^\text{\rm int}}

\newcommand{\B}{\mathcal{B}}

\newcommand{\e}{\mathsf{e}}

\title[Local rigidity of conformally Euclidean metrics]{Local rigidity of conformally Euclidean metrics for the anisotropic Calder\'on problem}

\author[S. Mu\~noz-Thon]{Sebasti\'an Mu\~noz-Thon}
\address{Universit\'e Paris-Saclay, Laboratoire de math\'ematiques d'Orsay, 91405, Orsay, France}
\email{sebastian.munoz-thon@universite-paris-saclay.fr}

\author[P. Stefanov]{Plamen Stefanov}
\thanks{Corresponding author: Plamen Stefanov.}
\address{Department of Mathematics, Purdue University, West Lafayette, IN 47907}
\email{stefanov@math.purdue.edu}

\DeclareMathOperator{\diver}{div}

\begin{document}
\begin{abstract}
We prove local rigidity of every smooth conformally Euclidean metric for the anisotropic Calderón problem on smooth compact domains $M\subset\mathbb R^n$, $n\ge3$. Any smooth Riemannian metric $g$ sufficiently close to a fixed background $g_0=e^{2c}\mathsf e$ in $H^s(M)$, with integer $s>n/2+1$,  having the same Dirichlet-to-Neumann map satisfies $g=\Phi^*g_0$ for a smooth diffeomorphism $\Phi$ fixing the boundary pointwise. 
\end{abstract}

\maketitle

\section{Introduction}\label{sec1} 
Let $M\subset\mathbb R^n$, $n\ge3$, be a smooth compact domain, and let $\mathsf e$ denote the Euclidean metric. For a smooth Riemannian metric $g$ on $M$, the Dirichlet-to-Neumann (DN) map is defined, for $f\in C^\infty(\bo)$, by
\[
\Lambda_g f=\left.\partial_{\nu_g}u\right|_{\bo},
\]
where  $u$ is the unique solution of the Dirichlet problem
\[
\begin{cases}
\Delta_g u=0&\text{in }M,\\
u=f&\text{on }\bo.
\end{cases}
\]
Here,  $\nu_g$ is the outward unit normal with respect to $g$. 
We prove  that for every smooth conformal factor $c$ in $M$, if $g$ is sufficiently close to $g_0\colon=e^{2c}\mathsf e$ in $H^s(M)$, with integer $s>n/2+1$, and if $\Lambda_g=\Lambda_{g_0}$, then $g$ is isometric to $g_0$ by a diffeomorphism fixing the boundary pointwise. We present the proof for $c=0$ first because it is simpler, and because its proof allows  the smallness parameter $\varepsilon_0>0$ to be bounded from below in an explicit way, in principle. 

\begin{theorem}
\label{thm1}
Let $M\subset\R^n$, $n\ge3$, be a smooth compact domain. Fix an integer $s>n/2+1$. There exists $\varepsilon_0>0$, depending only on $M$ and $s$, such that for every smooth Riemannian metric $g$ on $M$, the conditions
\be{1a}
\Lambda_g=\Lambda_\e,\qquad 
\|g-\e\|_{H^s(M)}<\varepsilon_0
\ee
imply that
\[
g=y^*\e
\]
for a  diffeomorphism $y\colon M\to M$ fixing $\bo$ pointwise. The map $x\mapsto y$ is given by the $g$-harmonic extensions
\[
\Delta_g y^j=0\quad\text{in $M$},
\qquad y^j|_\bo=x^j|_\bo,\qquad j=1,\dots,n.
\]
\end{theorem}

The idea is the following. 
We first put $g$ in harmonic coordinates fixing the boundary. Writing $\overline{g}=y_*g$, the tensor
\[
B=\sqrt{\det\overline{g}}\,\overline{g}^{-1}-\Id
\]
is then divergence-free (which reminds us of the solenoidal gauge in boundary rigidity but this time, it is about the conductivity tensor). Equality \r{1a} of the DN maps, together with boundary determination in \cite{LeeU}, implies that $\overline{g}-\mathsf e$ has a vanishing boundary jet. We can therefore extend $\overline{g}$ and $B$ smoothly outside $M$ by the Euclidean metric and zero, respectively.

For the Euclidean metric, we have the harmonic functions used by Calderón in his original paper  \cite{Calderon80}:
\begin{equation}\label{1}
u_{\mathsf e}(x,\zeta)=e^{x\cdot\zeta},\qquad \zeta\in \C^n, \; 
\zeta\cdot\zeta=0.
\end{equation}
Equality of the DN maps also allows us to construct  global $\overline{g}$-harmonic solutions of the form $e^{x\cdot\zeta}(1+r_\zeta)$, where $r_\zeta$ is smooth and supported in $M$, as shown in  Proposition~\ref{pr1} below. 
The remainder $r_\zeta$ is controlled when $|\zeta|\|B\|_{L^\infty} \ll1$ by a known  Carleman estimate. An orthogonality identity of Alessandrini type, combined with the divergence-free condition, then yields a quadratic estimate for   $\widehat B$ in a ball in the frequency space with radius at most $C_0 /\|B\|_{L^\infty}$, see \r{13}. We choose its radius as a fractional power of $\|B\|_{L^2}^{-1}$ and control the higher frequencies using the a priori $H^s$ bound supplied by the theorem’s smallness assumption \r{1a}. We use interpolation to ensure that this radius would lie within the allowed range, and smallness allows us to absorb both contributions, proving $B=0$.

We emphasize that we do not take $|\zeta|\to\infty$ for a fixed metric. Our solutions are $\overline{g}$-harmonic extensions of the boundary values of the Euclidean solutions \r{1}.
We keep  the Euclidean exponential factor $e^{x\cdot\zeta}$, although for a general anisotropic metric its phase does not satisfy the  eikonal equation $\overline{g}^{ij}\zeta_i\zeta_j=0$ in the metric. 
This representation does not determine the interior phase, since the factor $1+r_\zeta$ may contribute to it.
Instead of taking the usual $|\zeta|\to\infty$ asymptotics of the Sylvester–Uhlmann construction in \cite{SylvesterU87}, which force $r_\zeta \to0$ in their setting, we obtain in Proposition~\ref{pr1} an estimate for $r$ in terms of the perturbation, uniformly over the allowed frequency range.
Although this range grows as the perturbation tends to zero, the rigidity argument requires no large-frequency limit for a fixed nonzero perturbation.

Our main result, which implies Theorem~\ref{thm1} but with a different proof, is to allow $g_0$ to be any fixed smooth conformally Euclidean background metric.

\begin{theorem}\label{thm2}
Let $M\subset\R^n$, $n\ge3$, be a smooth compact domain, let $g_0=e^{2c}\e$ with $c\in C^\infty(M)$, and fix an integer $s>n/2+1$. There exists $\varepsilon_0>0$, depending on $M$, $g_0$, $s$, such that every smooth Riemannian metric $g$ on $M$ satisfying
\[
\|g-g_0\|_{H^s(M)}<\varepsilon_0
\]
and
\begin{equation}\label{15}
\Lambda_g=\Lambda_{g_0}
\end{equation}
has the form $g=\Phi^*g_0$ for a smooth diffeomorphism $\Phi\colon M\to M$ fixing $\partial M$ pointwise.
\end{theorem} 

The proof of Theorem~\ref{thm2} replaces the harmonic coordinates of Theorem~\ref{thm1} by 
a harmonic embedding in a higher-dimensional Euclidean space to construct the gauge. We  replace the Euclidean exponentials \r{1} by background C.G.O.\ solutions. 
For the conformal background $g_0=e^{2c}\mathsf e$, we use the normalized conductivity perturbation
\[
B=e^{-(n-2)c}\sqrt{\det\overline{g}}\,\overline{g}^{-1}-\Id,
\]
where $\overline{g}$ is the metric in the harmonic gauge. This agrees with the preceding definition when $c=0$, and for uniqueness it is enough to prove $B=0$. We do not have $\diver B=0$ anymore; instead, we prove that $\|\diver B\|_{L^2}\le C\|B\|_{L^2}$. 
Assuming that the uniqueness does not hold in this gauge, we get a normalized sequence of perturbations for which we establish compactness. Passing to a subsequence, we obtain a nonzero element of the kernel of the linearization. We show that this kernel consists of potential tensors, and that our choice of gauge forces the limiting tensor to vanish. This contradiction completes the argument. Since we prove Theorem~\ref{thm2} by a compactness argument, an explicit lower bound for $\varepsilon_0$ is not readily  provided.  

Theorem~\ref{thm2} yields the following corollary. 

\begin{corollary}\label{cor_q}
Let $\Lambda_{g,q}$ be the DN map associated with the operator $-\Delta_g+q$, where $q$ is a real-valued potential smooth in $M$. Let $g_0$ and $s$ be as in Theorem~\ref{thm2}, let $q_0$ be a smooth real-valued potential on $M$, and assume that the first Dirichlet eigenvalue of $-\Delta_{g_0}+q_0$ in $M$ is positive. There exists $\eps_0=\eps_0(M,g_0,q_0,s)>0$ so that when  
\begin{equation}   \label{c0}
\|g-g_0\|_{H^s(M)} + \|q-q_0\|_{H^{s-1}(M)}<\varepsilon_0
\end{equation}
 and $\Lambda_{g,q}=\Lambda_{g_0,q_0}$, then
\begin{equation}   \label{c1}
g=a^{\frac{4}{n-2}}\Phi^* g_0, \quad q= a^{-\frac{4}{n-2}}\left( \Phi^*q_0 - \frac{\Delta_{\Phi^* g_0}a}{a} \right)
\end{equation}
for some smooth function $a>0$ with $a=1$, $\d a=0$ on $\bo$, and a smooth diffeomorphism $\Phi$ fixing $\bo$ pointwise. 

On the other hand, for every $(g,q)$ satisfying  \r{c1}, and the boundary conditions, we have $\Lambda_{g,q}=\Lambda_{g_0,q_0}$. 
\end{corollary}

The gauge \r{c1} has been specified in \cite[Theorem~8.4]{DosSantosKSU09} in their boundary determination result.  We reduce Corollary~\ref{cor_q}  to Theorem~\ref{thm2} by  applying the Liouville transform. Clearly, the eigenvalue condition holds at least when $q_0\ge0$.

\subsection*{Previous work}
Calder\'on's problem has a rich history starting with the original work by Calder\'on \cite{Calderon80} who studied  isotropic conductivities  (corresponding to metrics conformal to $\e$ when $n\ge3$ in metrics language), and proved injectivity of the linearization using the solutions \r{1}. In their seminal paper \cite{SylvesterU87}, Sylvester and Uhlmann proved global uniqueness for the nonlinear problem in dimensions $n\ge3$, for isotropic conductivities, introducing complex geometric optics (C.G.O.) solutions. Those solutions created a powerful technique used in many other problems. In \cite{Nachman_98}, Nachman proposed a constructive reconstruction method. 

Since the paper \cite{SylvesterU87} appeared, it was conjectured that a similar result should hold for general anisotropic conductivities (general metrics in the metric language), up to an action of a diffeomorphism fixing $\partial M$ pointwise. 
The anisotropic version turned out to be much harder. 
The two-dimensional case is more approachable, but not easy by any means; and there one has conformal invariance as well, in addition to the diffeomorphism one. It has been solved, see \cite{Sylvester90,LassasU01,AstalaLP05, GuillarmouTzou11, CarsteaLT24}.

The $n\ge3$ anisotropic case is mostly open. The main difficulty, when trying to adapt the  method in \cite{SylvesterU87}, is the inability to construct C.G.O.\ solutions. 
Under an analyticity assumption of the metric class, uniqueness modulo the gauge has been established in increased generality in \cite{LeeU,LassasU01,LassasTU03,LLS-Poisson20}. Carleman estimates and complex geometric optics constructions have yielded uniqueness results for conformal factors on   conformally transversally anisotropic manifolds, which have conformally product types of metrics. Dos Santos Ferreira, Kenig, Salo, and Uhlmann \cite{DosSantosKSU09} characterized the geometric restrictions imposed by the existence of limiting Carleman weights and constructed complex geometric optics solutions on admissible manifolds, obtaining uniqueness within a fixed conformal class. The approach was extended to more general transversal geometries in \cite{DosSantosKLS16}. The results described above impose additional analyticity or geometric restrictions on the metrics. 
Even the linearization about a fixed anisotropic metric is not known to be injective modulo the linearized gauge. Even if it were, the inherent instability of the problem \cite{Alessandrini88, Mandache} would not yield local uniqueness by linearization, see also \cite{SU-JFA09}. Here we prove local rigidity near each fixed smooth conformally Euclidean metric,  allowing arbitrary (sufficiently small) anisotropic perturbations.

Our first  theorem belongs to a family of rigidity results saying that a metric which looks flat from boundary measurements must be flat, up to the natural diffeomorphism gauge. In boundary rigidity, Gromov \cite{Gromov}, among many other results, established rigidity of Euclidean regions. A Lorentzian counterpart was obtained by Oksanen, Rakesh, and Salo \cite{ORS-Lorentzian26}: a  globally hyperbolic metric agreeing with the Minkowski metric outside a compact set and having the same hyperbolic DN map is isometric to the Minkowski metric. Those results have been generalized.  Burago and Ivanov \cite{Burago-Ivanov} proved boundary rigidity for metrics sufficiently close to Euclidean ones. In subsequent work \cite{ORS-Semiglobal26}, Oksanen, Rakesh, and Salo extended their rigidity result to semiglobal uniqueness, allowing the background metric to be a sufficiently small perturbation of the Minkowski metric. Theorem~\ref{thm2} establishes local rigidity for each fixed smooth conformally Euclidean background and nearby anisotropic metrics.

\subsection*{Acknowledgments}
The authors used ChatGPT (OpenAI) for mathematical discussions, assistance with   checking parts of the proof, and checking bibliographic references. The authors take responsibility for the final mathematical statements and proofs. Thanks are due to Ali Feizmohammadi for attracting P.S.'s attention to   \cite[Theorem~5.3]{Salo08}, and to him, Katya Krupchyk, and  Gunther Uhlmann for the numerous discussions on  Calder\'on's problem. S.~M.-T. thanks Thibault Lefeuvre for critical comments while working on this project.

\subsection*{Note}

After the first version of this paper, containing Theorem~\ref{thm1},
was posted, Yi-Hsuan Lin informed P.S.\ that he had independently obtained the same result as that theorem by a related argument. Lin's preprint was
subsequently posted as \cite{Lin26}.

\section{Proof of Theorem~\ref{thm1}}\label{sec2}
All metrics, functions, and tensors in the rest of the paper are smooth unless stated otherwise.
We work in this section under the assumptions of Theorem~\ref{thm1}, choosing $\varepsilon_0$ sufficiently small as required below. We write $h=g-\e$ for the original metric perturbation. Sobolev norms are taken in a Euclidean way. For tensor fields, we use the square root of the sum of the squared component Sobolev norms, and pointwise norms are Frobenius norms. 

\subsection{The harmonic gauge}\label{sec2a} 
The idea to use harmonic coordinates was suggested by John Lee to the authors of \cite{uhl_syl}; harmonic coordinates were also used in \cite{SU-JFA}.  
Harmonic coordinates were also used in the Poisson embedding approach in \cite{LLS-Poisson20}. In two dimensions, they were also used in \cite{AlessandriniCabib07}, and in \cite{DDO13}.

In harmonic coordinates $y=(y^1,\dots,y^n)$, the Laplacian has no first-order terms. Indeed, write $\overline{g}=y_*g$ for the metric in these coordinates. The equations $\Delta_g y^j=0$ imply $\partial_{y^i}\big(\sqrt{\det\overline{g}}\,\overline{g}^{ij}\big)=0$, and therefore
\[
\Delta_{\overline{g}}
=\frac{1}{\sqrt{\det\overline{g}}}\partial_{y^i}\big(\sqrt{\det\overline{g}}\, \overline{g}^{ij}\partial_{y^j}\big)
=\overline{g}^{ij}\partial_{y^i}\partial_{y^j}.
\]
We next construct such coordinates globally and establish their compatibility with the boundary jets.

\begin{lemma} \label{lemma1}
Let $y$ be the harmonic extension map defined in Theorem~\ref{thm1}. For $\varepsilon_0>0$ sufficiently small, $y\colon M\to M$ is a  diffeomorphism fixing $\bo$ pointwise, and
\[
\|y-\Id\|_{H^{s+1}(M)}+\|y_*g-\e\|_{H^s(M)}
\le C\|h\|_{H^s(M)}.
\]
Here $C$ depends  on $M$ and $s$ only.
\end{lemma}

\begin{proof}
Set $w^j(x)=y^j(x)-x^j$, $j=1,\dots,n$, and $A^{ij}=\sqrt{\det g}\,g^{ij}$, where $x$ are the original coordinates. The equation for $w^j$ is, with $\partial_i \colon= \partial_{x^i}$, 
\[
\Delta w^j=-\partial_i(A^{ij}-\delta^{ij})
-\partial_i\big((A^{ik}-\delta^{ik})\partial_k w^j\big),
\qquad w^j|_\bo=0.
\]
A standard estimate for the Dirichlet realization of $\Delta$, combined with the fact that $H^s(M)$ is an algebra for $s>n/2$, yields
\[
\|w\|_{H^{s+1}(M)}
\le C\|A-\Id\|_{H^s(M)}\big(1+\|w\|_{H^{s+1}(M)}\big).
\]
We have $\|A-\Id\|_{H^s(M)}\le C\|h\|_{H^s(M)}$ with $C$ depending only on $M$ and $s$, uniformly for $\varepsilon_0$ sufficiently small. By the smallness assumption in \r{1a}, choosing $\varepsilon_0$ sufficiently small allows us to absorb $w$ on the right and prove the estimate for it. Decreasing $\varepsilon_0$ further if necessary, we get  $\|D w\|_{L^\infty(M)}\le 1/2$ by  Sobolev embedding.

We will show that $x\mapsto y$ is a diffeomorphism on $M$ now. 
Set $d(x)=\dist(x,\bo)$. Integrating $Dw$ along a segment to a nearest boundary point yields $|w(x)|\le \frac12 d(x)$, thus $y(x)=x+w(x)$ maps $M$ into itself. To prove surjectivity, for $z\in\Mint$, consider the map $T_z(x)=z-w(x)$. We want to solve $x=T_z(x)$,  equivalent to $x=z-w(x)$, uniquely. The map $T_z$ is a contraction on the closed ball $\overline{B(z,d(z))}\subset M$, mapping that ball into itself because
\[
|T_z(x)-z|=|w(x)|\le \frac12 d(x)\le \frac12\big(d(z)+|x-z|\big)
\le d(z).
\]
Moreover, every solution of $y(x)=z$ lies in this ball, since
\[
|x-z|=|w(x)|\le \frac12 \bigl(d(z)+|x-z|\bigr)
\quad\Longrightarrow\quad
|x-z|\le d(z).
\]
Banach's contraction theorem therefore implies that   every interior point has exactly one preimage. Together with the condition $y|_\bo=\Id$, this proves that $x\mapsto y$ is a bijection. Since $y$ is smooth up to the boundary, and $Dy=\Id+Dw$ is everywhere invertible, the inverse function theorem, including at boundary points, shows that $y^{-1}$ is smooth up to the boundary. Thus $y$ is a  diffeomorphism of $M$ onto itself.

Since $y$ is close to $\Id$ in $H^{s+1}(M)$ and $\det(Dy)\ge2^{-n}$, the Sobolev inversion, composition, and product estimates in \cite{IKT13} yield
\[
\|y_*g-\e\|_{H^s(M)}
\le C\bigl(\|h\|_{H^s(M)}+\|w\|_{H^{s+1}(M)}\bigr)
\le C\|h\|_{H^s(M)},
\]
with $C$ depending only on $M,s$.
\end{proof}

Set $\overline{g}=y_*g$. In the next lemma, ``flat'' means zero jet at $\bo$, not to be confused with zero curvature.

\begin{lemma}\label{lemma2}
The metric perturbation $\overline{g}-\e$ is flat at $\bo$.
\end{lemma}

\begin{proof}
We use the boundary recovery result  \cite[Propositions~1.1--1.2]{LeeU}. The  DN map there $f\mapsto(\Lambda_gf)\, \d S_g$ maps functions to forms. 
It implies that the principal symbol of $\Lambda_g$ with our definition is $\sigma_1(\Lambda_g)(x',\xi')=|\xi'|_{g|_{T\bo}}$. Thus $\Lambda_g=\Lambda_{g_0}$ implies $g|_{T\bo}=g_0|_{T\bo}$. This allows us to remove the determinant factor incorporated in $\d S_g$. 
The smooth boundary determination result in \cite[Proposition~1.3]{LeeU} therefore recovers the full boundary jet of the metric in boundary normal coordinates from our DN map. Note that even though \cite{LeeU} is about analytic metrics, the proof there requires no analyticity assumptions. 
 By the DN equality in \r{1a}, there is therefore a diffeomorphism $\psi$ defined in a boundary collar, fixing $\bo$ pointwise, such that $\widehat g\colon=\psi^*g$ and $\e$ have the same boundary jets.  
Set $Y=y\circ\psi$ in the collar. Clearly, $\Delta_{\widehat g}Y^j=(\Delta_g y^j)\circ\psi=0$, and having the same DN maps leads to
\[
\partial_{\nu_{\widehat g}}Y^j=\partial_{\nu_g}y^j=\Lambda_g(x^j|_\bo)
=\Lambda_\e(x^j|_\bo) =\partial_{\nu_\e}x^j \quad\text{on $\bo$}.
\]
Since $\widehat g=\e$ at $\bo$, the unit normals of those two metrics agree there. Thus $w^j\colon=Y^j-x^j$ satisfies
\begin{equation}   \label{2}
\Delta_{\widehat g}w^j=(\Delta_\e-\Delta_{\widehat g})x^j
\quad\text{in the collar},\qquad
w^j|_\bo=\partial_{\nu_{\widehat g}} w^j|_\bo=0.
\end{equation}
The r.h.s.\ of the PDE in \r{2} is flat at $\bo$. In Euclidean boundary normal coordinates, the vanishing of $w^j$ and its first normal derivative implies the vanishing of all their tangential derivatives at the boundary. Problem \r{2} is non-characteristic, so the equation implies that the second normal derivative also vanishes. Repeated normal differentiation of the equation then shows inductively that all higher normal derivatives vanish: at each step, the remaining terms involve only tangential derivatives of normal derivatives already known to vanish. Thus $w^j$ is flat at $\bo$ for every $j$, and $Y-\Id$ has zero boundary jet.

Near $\bo$, we can pass back to the harmonic coordinates $y$ to get  $Y_*\widehat g=y_*g=\overline{g}$. Since $\widehat g-\e$ and $Y-\Id$ are flat at $\bo$, so is $\overline{g}-\e$. 
\end{proof}

\begin{lemma}\label{lemma3}
The metric
\[
\tilde g=
\begin{cases}
\overline{g}=y_*g,&\text{in }M,\\
\e,&\text{in }\R^n\setminus M
\end{cases}
\]
is smooth on $\R^n$, and its Cartesian coordinate functions are $\tilde g$-harmonic.
\end{lemma}

\begin{proof}
Smoothness follows from Lemma~\ref{lemma2}.  Next, 
\[
\Delta_{\overline{g}}x^j=(\Delta_g y^j)\circ y^{-1}=0
\quad\text{in }M.
\]
Outside $M$, the metric is Euclidean, so $\Delta_{\tilde g}x^j=0$ there as well. Since $\tilde g$ and $x^j$ are smooth on $\R^n$, this identity holds across $\bo$.
\end{proof}

\subsection{Main identity}\label{sec2b}
For the remainder of this section, we denote $\tilde g$ by $g$. We have $g$ defined in $\R^n$ now, with its Euclidean coordinates $g$-harmonic. Since $y$ fixes $\bo$ pointwise, the DN equality in \r{1a} continues to hold for this metric. The smallness assumption in \r{1a} refers to the original perturbation $h$. We use the notation
\begin{equation}   \label{2a}
A^{ij}= \sqrt{\det g} g^{ij}, \quad B^{ij}= A^{ij}-\delta^{ij}
\end{equation}
introduced earlier, now used throughout $\R^n$. 
Then  we have in $\R^n$, 
\begin{equation}   \label{3}
\partial_i B^{ij}=0, \quad \supp B\subset M.
\end{equation}
Also, 
\begin{equation}   \label{4}
 \sqrt{\det g} \,\Delta_g = \Delta + B^{ij}\partial_i \partial_j. 
\end{equation}
Integrating by parts and using the DN equality in \r{1a}, we get Alessandrini's identity \cite{Alessandrini88} in this case:
\begin{equation}   \label{5}
\int B^{ij} (\partial_i u_g)\,(\partial_j v_\e)\,\d x=0
\end{equation}
for all $g$-harmonic $u_g$  and $\e$-harmonic  $v_\e$, both in $C^1(M)$. By \r{3}, the integral in \r{5} extends to $\R^n$. 

Since $B$ is divergence-free, an integration by parts in \r{5} yields
\begin{equation}   \label{6}
\int_M B^{ij} u_g\partial_i  \partial_j v_\e\,\d x=0.
\end{equation}
Plugging the exponential solutions \r{1} for $v_\e$, we get
\begin{equation}   \label{7}
\int_M e^{x\cdot\zeta} B^{ij}(x) \zeta_i \zeta_j  u_g(x)\,\d x=0, \quad \zeta\cdot\zeta=0,
\end{equation}
for every $C^1(M)$ $g$-harmonic function $u_g$. This equality is our main orthogonal identity which we will use to prove that $B=0$. Once we do that, taking determinants in $\sqrt{\det g} g^{-1}=\Id$ yields
\[
1=\det\bigl(\sqrt{\det g}\,g^{-1}\bigr)=(\det g)^{(n-2)/2}.
\]
Since $n\ge3$, it follows that $\det g=1$, and then $g=\e$.

\subsection{Construction of C.G.O.-like solutions related to $g$} We use the C.G.O.\ ansatz as in \cite{SylvesterU87}. As explained in the Introduction, we construct solutions with frequencies bounded by $C/\|B\|_{L^\infty}$, assuming $B\not\equiv0$. To reiterate, they \textit{are not} extensions of the  C.G.O.\ solutions in \cite{SylvesterU87}  to anisotropic metrics. 

For $\zeta\not=0$ such that $\zeta\cdot\zeta=0$, write, in ``polar coordinates,'' 
\be{7a}
\zeta=\tau(a+\i b)\quad  \text{with $a\perp b$ real and unit, and $\tau>0$.} 
\ee
 Conjugate $\Delta_\e$ with $e^{\zeta\cdot x}$ to get 
\[
P_\zeta\colon= e^{-\zeta\cdot x}  \Delta_\e e^{\zeta\cdot x} = \Delta_\e+2\zeta\cdot\nabla .
\]
We write $\langle\tau\rangle=(1+\tau^2)^{1/2}$, $\langle\zeta \rangle=(1+|\zeta|^2)^{1/2}$ in what follows. 

\begin{lemma}\label{lemma4} 
For $v\in H^2(\R^n)$ supported in a fixed ball, and for every $\tau>0$, we have
\begin{equation}   \label{10}
\|v\|_{L^2}+\langle\tau\rangle^{-1}\|\nabla v\|_{L^2}
+\langle\tau\rangle^{-2}\|\nabla^2v\|_{L^2}
\le C\langle\tau\rangle^{-1}\|P_\zeta v\|_{L^2}.
\end{equation}
All norms are over $\R^n$, and $C$ depends only on the fixed ball and $n$.
In particular,
\[
\|v\|_{H^2}\le C\|P_\zeta v\|_{L^2},\qquad 0<\tau\le1.
\]
\end{lemma}

\begin{proof}
Choose a fixed ball whose interior contains the prescribed support ball. It suffices to consider smooth $v$ compactly supported in this larger ball. We use the Carleman estimate in \cite[Theorem~4.1]{DosSantosKSU09}, see also \cite{Gunther-book} and \cite[Theorem~5.3 and its proof]{Salo08}, applied to $e^{\i\tau b\cdot x}v$ with  the linear Carleman weight $-a\cdot x$ and small parameter $h=\tau^{-1}$, to get
\begin{equation}   \label{10v}
\|v\|_{L^2}\le C\tau^{-1}\|P_\zeta v\|_{L^2}.
\end{equation}
Here the potential   is zero, so this estimate holds for every $\tau>0$. For completeness, we provide the short proof in this case following the proof of \cite[Theorem~5.3]{Salo08}. Set $w=e^{\i\tau b\cdot x}v$. Then
\[
e^{\i\tau b\cdot x}P_\zeta v
=(\Delta+\tau^2+2\tau a\cdot\nabla)w.
\]
The self-adjoint operator $\Delta+\tau^2$ commutes with the skew-adjoint operator $2\tau a\cdot \nabla$, thus integration by parts yields
\[
\|P_\zeta v\|_{L^2}^2=\|(\Delta+\tau^2)w\|_{L^2}^2+4\tau^2\|a\cdot \nabla w\|_{L^2}^2.
\]
The Poincar\'e inequality along lines parallel to $a$ gives $\|a\cdot\nabla w\|_{L^2}^2\ge C^{-1}\|w\|_{L^2}^2=C^{-1}\|v\|_{L^2}^2$, with $C$ depending only on the fixed ball. This proves \r{10v}.

The constant in \r{10v} is independent of the orthonormal pair $a,b$, by rotational invariance of the Euclidean Laplacian. 
For $0<\tau\le1$, set $u=e^{x\cdot\zeta}v$, so that
$\Delta u=e^{x\cdot\zeta}P_\zeta v$. The functions $e^{\pm x\cdot\zeta}$ are uniformly bounded on the chosen ball in this range. Let $\lambda_1>0$ be the first eigenvalue of the Dirichlet Laplacian $-\Delta$ on that ball. Since $u$ vanishes near its boundary, the spectral theorem yields
\[
\|v\|_{L^2}\le C\|u\|_{L^2}
\le C\lambda_1^{-1}\|\Delta u\|_{L^2}
\le C\|P_\zeta v\|_{L^2}.
\]
Combining the two bounds, we obtain, for all $\tau>0$,
\[
\|v\|_{L^2}\le C\langle\tau\rangle^{-1}\|P_\zeta v\|_{L^2}.
\]
Integration by parts or using the Fourier transform  yields
\begin{equation}   \label{10b}
\|\nabla v\|_{L^2}^2\le \|v\|_{L^2}\|\Delta v\|_{L^2},
\qquad
\|\nabla^2v\|_{L^2}=\|\Delta v\|_{L^2}.
\end{equation}
Since $\Delta v=P_\zeta v-2\zeta\cdot\nabla v$ and $|\zeta|=\sqrt2\tau$,  
\[
\begin{split}
\|\Delta v\|_{L^2}
&\le \|P_\zeta v\|_{L^2}+2\sqrt2\tau\|\nabla v\|_{L^2}\\
&\le \|P_\zeta v\|_{L^2}+\frac12\|\Delta v\|_{L^2}
+4\tau^2\|v\|_{L^2},
\end{split}
\]
where we used \r{10b} and the  inequality $ab\le (a^2+b^2)/2$. 
Absorbing the Laplacian term and using the $L^2$ estimate for $v$, we obtain, for all $\tau>0$,
\[
\|\nabla^2v\|_{L^2}=\|\Delta v\|_{L^2}
\le C\big(\|P_\zeta v\|_{L^2}+\tau^2\|v\|_{L^2}\big)
\le C\langle\tau\rangle\|P_\zeta v\|_{L^2}.
\]
Inequality \r{10b} then implies
\[
\|\nabla v\|_{L^2}^2
\le \|v\|_{L^2}\|\Delta v\|_{L^2}
\le C\|P_\zeta v\|_{L^2}^2.
\]
Combining these estimates proves \r{10}.
\end{proof}

We construct our $g$-harmonic solutions now. We continue to use the convention $|\zeta|= \sqrt2 \tau$.

\begin{proposition}\label{pr1}
For $\zeta$ such that $\zeta\cdot\zeta=0$,  let $u$ solve
\begin{equation}   \label{8}
\Delta_g u=0\quad \text{in $M$}, \qquad u|_{\bo}=e^{x\cdot\zeta}. 
\end{equation}
Then there exist constants $C_0>0$, $C>0$, depending only on $M$ and $n$, so that $u=e^{x\cdot\zeta}(1+r)$ with
\begin{equation}   \label{9}
\|r\|_{L^2(M)}+\langle\tau\rangle^{-1}\|\nabla r\|_{L^2(M)}
+\langle\tau\rangle^{-2}\|\nabla^2r\|_{L^2(M)}
\le C\frac{\tau^2}{\langle\tau\rangle}\|B\|_{L^2(M)}
\end{equation}
when 
\[
\|B\|_{L^\infty}<C_0,\qquad
0<\tau<C_0/\|B\|_{L^\infty}.
\]
Moreover, extending $u$ by $e^{x\cdot\zeta}$ outside $M$ results in a smooth $g$-harmonic function $U$ on $\R^n$. The corresponding zero extension of $r$ satisfies $r=e^{-x\cdot\zeta}U-1\in C_0^\infty(\R^n)$ and $\supp r\subset M$.
\end{proposition}

\begin{proof}
Set $E(x)=e^{x\cdot\zeta}$, so that $\Delta E=0$. Since $g=\e$ on $\bo$, the unit normals for the two metrics agree there. By the DN equality in \r{1a}, we get 
\[
\partial_\nu u=\Lambda_g(E|_\bo)
=\Lambda_\e(E|_\bo)=\partial_{\nu_\e} E
\quad\text{on $\bo$}.
\]
Extend $u$ to $U$ on $\R^n$ by setting $U=E$ outside $M$. The conormal flux $\sqrt{\det g}\,g^{ij}\nu_i\partial_j u$ equals $\partial_\nu u$ on $\bo$, since $g=\e$ there. Thus the boundary values and conormal fluxes match across $\bo$, so integration by parts shows that $\Delta_gU=0$ in the distributional sense in $\R^n$. Elliptic regularity implies that $U$ is smooth. Therefore
\[
r=e^{-x\cdot\zeta}U-1\in C_0^\infty(\R^n),
\qquad \supp r\subset M.
\]

Conjugating \r{4} by $E$ and using $\zeta\cdot\zeta=0$ and the symmetry of $B$, we obtain on $\R^n$
\begin{equation}\label{11}
P_\zeta r=-B^{ij}\zeta_i\zeta_j
-B^{ij}\big(\partial_i\partial_jr+2\zeta_i\partial_jr+\zeta_i\zeta_jr\big).
\end{equation}
Let $X_\tau(r)$ be the norm on the left-hand side of \r{9}, 
where all norms are over $\R^n$. Since $|\zeta|=\sqrt2\tau$, equation \r{11} implies
\begin{equation}   \label{14a}
\begin{split}
\|P_\zeta r\|_{L^2}
&\le C\tau^2\|B\|_{L^2}
+C\|B\|_{L^\infty}\big(\|\nabla^2r\|_{L^2}
+\tau\|\nabla r\|_{L^2}+\tau^2\|r\|_{L^2}\big)\\
&\le C\tau^2\|B\|_{L^2}
+C\langle\tau\rangle^2\|B\|_{L^\infty}X_\tau(r).
\end{split}
\end{equation}
Apply Lemma~\ref{lemma4} in a fixed ball containing $M$ to the solution $r$ of \r{11}. For every $\tau>0$ it follows that
\begin{equation}   \label{14b}
\begin{split}
X_\tau(r)\le C\langle\tau\rangle^{-1}\|P_\zeta r\|_{L^2} \le C\frac{\tau^2}{\langle\tau\rangle}\|B\|_{L^2}
+C\langle\tau\rangle\|B\|_{L^\infty}X_\tau(r).
\end{split}
\end{equation}
Choose $C_0>0$ so that $2C C_0\le1/2$. If $\|B\|_{L^\infty}<C_0$ and $\tau\|B\|_{L^\infty}<C_0$, required by Proposition~\ref{pr1}, then
\[
C\langle\tau\rangle\|B\|_{L^\infty}<2CC_0\le\frac12.
\]
Thus the last term can be absorbed into the left-hand side, yielding
\[
X_\tau(r)\le 2C\frac{\tau^2}{\langle\tau\rangle}\|B\|_{L^2}.
\]
Both $r$ and $B$ are supported in $M$, so their norms over $\R^n$ equal their norms over $M$. This proves \r{9}.
\end{proof}

By Lemma~\ref{lemma1} and the Sobolev product and embedding estimates, choosing $\varepsilon_0$ in Theorem~\ref{thm1} sufficiently small ensures $\|B\|_{L^\infty}<C_0$, which we use for the rest of the argument.

We return to the orthogonal identity \r{7}. Let $u_\zeta=e^{x\cdot\zeta}(1+r_\zeta)$ be the solution provided by Proposition~\ref{pr1}. Choose $\eta$ with $\eta\cdot\eta=0$ so that $\zeta+\eta=-\i\xi$, $\xi\in\R^n$; see \r{eta} below for a specific choice. Taking $u_\e=e^{x\cdot\eta}$, we get 
\begin{equation}   \label{12}
\int_M e^{-\i x\cdot\xi} B^{ij}(x) \eta_i \eta_j  \,\d x= 
-\int_M e^{-\i x\cdot\xi} B^{ij}(x) \eta_i \eta_j  r_\zeta (x)\,\d x.
\end{equation}
Since $\Re\eta=-\Re\zeta$ and both vectors are null, $|\eta|=|\zeta|=\sqrt2\tau$. By the  Cauchy inequality, 
\[
\begin{split}
\left|\int_M e^{-\i x\cdot\xi}B^{ij}(x)\eta_i\eta_j\,\d x\right|
&\le |\eta|^2\|B\|_{L^2(M)}\|r_\zeta\|_{L^2(M)}.
\end{split}
\]
Using \r{9} and $\tau^2/\langle\tau\rangle\le\tau$, we obtain, under the conditions on $\xi$, $\eta$ above, 
\begin{equation}   \label{12a}
|\widehat B^{ij}(\xi)\eta_i \eta_j  |\le C|\eta|^3 \|B\|_{L^2(M)}^2 
\qquad \text{for} \quad 0<|\eta|<\frac{\sqrt2\,C_0}{\|B\|_{L^\infty}}.
\end{equation}

\subsection{Estimating the full  tensor}
The divergence-free condition \r{3} implies $\widehat B(\xi)\xi=0$. Our first goal is to combine this with polarization identity techniques to estimate all components of $\widehat B$ in \r{12a}. 

Write $\xi=\rho\omega$, $|\omega|=1$. Suppose $\rho>0$ and take $\tau=\rho$. Following a now standard construction \cite{SylvesterU87}, choose orthonormal $a,b\in\omega^\perp$, which is possible since $n\ge3$, and set
\be{eta}
\begin{aligned}
\eta_\pm&=\rho\bigg(-a+\i\bigg(-\frac{\omega}{2}
\pm\frac{\sqrt3}{2}b\bigg)\bigg),\quad 
\zeta_\pm&=\rho\bigg(a+\i\bigg(-\frac{\omega}{2}
\mp\frac{\sqrt3}{2}b\bigg)\bigg).
\end{aligned}
\ee
These vectors satisfy $\zeta_\pm+\eta_\pm=-\i\xi$ and $\eta_\pm\cdot\eta_\pm=\zeta_\pm\cdot\zeta_\pm=0$, with the parameter $\tau=\rho$.  Using the  divergence-free condition  $\widehat B \omega=0$, we get 
\[
\widehat B(\eta_\pm/\rho,\eta_\pm/\rho)
=\widehat B(a,a)-\frac34\widehat B(b,b)\mp\i\sqrt3\,\widehat B(a,b),
\]
where, for fixed $\xi$, we regard $\widehat B$ as the complex symmetric bilinear form $\widehat B(v,w)\colon=\widehat B^{ij}(\xi)v_iw_j$. 
For each fixed $\xi=\rho\omega\ne0$ in the allowed range, $0<|\xi|=\rho< C_0/ \|B\|_{L^\infty}$, 
we solve the following system for the entries of $\widehat B(\xi)$. Since $|\eta_\pm|=\sqrt2\rho$, dividing \r{12a} by $\rho^2$ yields
\[
\left|\widehat B(\eta_\pm/\rho,\eta_\pm/\rho)\right|
\le C\rho\|B\|_{L^2(M)}^2,
\]
uniformly in $\rho$, $\omega$, and the orthonormal pair $a,b$. Taking the sum and difference estimates $\widehat B(a,a)-\frac34\widehat B(b,b)$ and $\widehat B(a,b)$ by the same bound. Exchanging $a,b$ also estimates $\widehat B(b,b)-\frac34\widehat B(a,a)$. The system for the two diagonal entries has a constant coefficient non-singular  matrix. 
Its inverse is therefore bounded independently of $\xi$ and the choice of $a,b$. Taking pairs from an orthonormal basis of $\omega^\perp$, and using $\widehat B(\xi)\omega=0$, bounds all entries of $\widehat B(\xi)$ in an orthonormal basis of $\R^n$. Since the Frobenius tensor norm is invariant under orthogonal changes of basis, we obtain
\be{13}
|\widehat B(\xi)|\le C|\xi|\|B\|_{L^2(M)}^2
\quad\text{when }0<|\xi|<C_0/\|B\|_{L^\infty}
\ee
(but we can also use the component-wise norm here). 
The constant in \r{13} is uniform as $\xi\to0$: the remainder estimate \r{9} is uniform for small $\tau=\rho$, and the system above introduces no further dependence on $\rho$. By continuity, \r{13} extends to $\xi=0$, giving $\widehat B(0)=0$.

If we ignore for a moment the different norms of $B$ appearing above, this estimate is not very useful at its limiting frequency $|\xi|= C_0/\|B\|_{L^\infty}$. It does imply $\widehat B(0)=0$ however and is meaningful at intermediate frequencies. What we do below is to use a frequency  cutoff equal to a fractional power of $C/\|B\|_{L^2}$, and control the higher frequencies by the a priori bound on $\|B\|_{H^s}$.

\subsection{Splitting the frequencies}
Since $B$ is flat at $\bo$, its extension as zero is smooth. For the integer $s$ fixed in Theorem~\ref{thm1}, the Sobolev product estimates and Lemma~\ref{lemma1} therefore give
\[
\|B\|_{H^s(\R^n)}
\le C\|\overline{g}-\e\|_{H^s(M)}
\le C\|h\|_{H^s(M)},
\]
where $C$ is uniform for $\varepsilon_0$ sufficiently small. Thus, by \r{1a}, we may set
\[
m=\|B\|_{L^2},\qquad \text{then}\quad 
m\le\|B\|_{H^s}\le C'\eps_0,
\]
where $C'>0$ depends only on $M$ and $s$. All norms below are over $\R^n$. Assume $m>0$, since otherwise $B=0$ and there is nothing to prove. Since $s-1>n/2$, by Sobolev embedding and interpolation, 
\[
\|B\|_{L^\infty}
\le C\|B\|_{H^{s-1}}
\le C\|B\|_{L^2}^{1/s}\|B\|_{H^s}^{1-1/s}
\le Cm^{1/s}\eps_0^{1-1/s}.
\]
Thus any $R$ satisfying
\[
0<R<c\,m^{-1/s}\varepsilon_0^{-(1-1/s)},
\]
where $c>0$ depends only on $M,s$ and is sufficiently small, satisfies $R\|B\|_{L^\infty}<C_0$. Hence \r{13} holds for $|\xi|\le R$ for such $R$. 
We split the frequencies at $|\xi|=R$ now; for a similar argument, see \cite{Alessandrini88} but the difference here is that estimate \r{13} is quadratic in $B$:
\[
m^2=(2\pi)^{-n}
\bigg(\int_{|\xi|\le R}|\widehat B(\xi)|^2\,\d\xi
+\int_{|\xi|>R}|\widehat B(\xi)|^2\,\d\xi  \bigg).
\]
For the low frequencies, by \r{13},
\[
(2\pi)^{-n}\int_{|\xi|\le R}|\widehat B(\xi)|^2\,\d\xi
 \le Cm^4\int_{|\xi|\le R}|\xi|^2\,\d\xi \le CR^{n+2}m^4.
\]
For the high frequencies, the a priori Sobolev bound implies
\[
\begin{split}
(2\pi)^{-n}\int_{|\xi|>R}|\widehat B(\xi)|^2\,\d\xi
 \le R^{-2s}(2\pi)^{-n}
\int_{\R^n}|\xi|^{2s}|\widehat B(\xi)|^2\,\d\xi \le CR^{-2s}\|B\|_{H^s}^2
\le C\eps_0^2R^{-2s}.
\end{split}
\]
Combining the two estimates, we obtain
\begin{equation}\label{14}
m^2\le CR^{n+2}m^4+C\eps_0^2R^{-2s},
\qquad 0<R<c\,m^{-1/s}\eps_0^{-(1-1/s)},
\end{equation}
where $C$ and $c$ are independent of $R$, $m$, and $\varepsilon_0$.

\subsection{Optimizing the frequency cutoff and completing the proof}
It remains to choose an admissible cutoff $R$ in \r{14} and complete the proof. 
Choose
\[
R=L(\eps_0/m)^{1/s},
\]
where $L\ge\max\{1,(4C)^{1/(2s)}\}$ is fixed, with $C$ the constant in \r{14}. 
The second  term in \r{14} satisfies
\[
C\eps_0^2R^{-2s}=CL^{-2s}m^2\le\frac14m^2.
\]
Since $m\le C'\eps_0$ and $2-(n+2)/s>0$ by the choice of $s$ in Theorem~\ref{thm1}, the first term in \r{14} satisfies, enlarging $C$ if necessary,
\[
\begin{split}
CR^{n+2}m^4 =CL^{n+2}\eps_0^{(n+2)/s}m^{\,2-(n+2)/s}m^2\le CL^{n+2}\eps_0^2m^2.
\end{split}
\]
For this choice of $R$, the upper bound in the second inequality in  \r{14} is equivalent to $L\eps_0<c$.
We first fix $L$ as above and then choose $\varepsilon_0$ sufficiently small that
\[
CL^{n+2}\varepsilon_0^2\le\frac14,
\qquad
L\varepsilon_0<c,
\]
where $C$ is the constant in the low-frequency estimate. The cutoff is then allowed in \r{14}, and both terms in \r{14} are at most $m^2/4$. Hence \r{14} yields  $m^2\le m^2/2$, contradicting $m>0$. Thus $B=0$ and $g=\e$ in the harmonic gauge, completing the proof of Theorem~\ref{thm1}.

\section{Proof of Theorem~\ref{thm2}}
\label{sec3}

\subsection*{Preliminaries}
We fix $M,g_0,s$ as in Theorem~\ref{thm2} and assume that $g$ is a smooth Riemannian metric satisfying \r{15} and sufficiently close to $g_0$ in $H^s(M)$.

For a symmetric matrix field $A$, write $L_Au=\partial_i(A^{ij}\partial_j u)$. The conductivity associated with a metric $g$ is
\begin{equation}\label{16}
A_g=\sqrt{\det g}\,g^{-1},\qquad L_{A_g}=\sqrt{\det g}\,\Delta_g.
\end{equation}
For the background metric $g_0=e^{2c}\e$, set
\begin{equation}\label{17}
a=e^{(n-2)c/2},\qquad A_0\coloneq A_{g_0}=a^2\Id.
\end{equation}
Extend $a$ smoothly to a positive function on $\mathbb R^n$, equal to one outside a fixed ball, and extend $g_0$ by $a^{4/(n-2)}\mathsf e$. We continue to denote the resulting positive extension by $a$, and extend $g_0$ by $a^{4/(n-2)}\e$. Then $a=1$ outside a fixed ball.

As in the proof of Lemma~\ref{lemma2}, the DN equality \r{15} implies equality of the induced boundary metrics. The integration by parts used in Section~\ref{sec2b} therefore yields
\begin{equation}\label{19}
\int_M(A_g-A_0)^{ij}(\partial_i u)\,(\partial_j v)\,\d x=0
\end{equation}
for $u,v\in C^\infty(M)$ satisfying $\Delta_g u=0$ and $\Delta_{g_0}v=0$ in $M$.

\subsection{The harmonic embedding}\label{sec3a}
In Section~\ref{sec2a}, the identity map provides global harmonic coordinates for the Euclidean background, and Lemma~\ref{lemma1} constructs the gauge by taking their $g$-harmonic extensions in $M$. For a general background $g_0$, we use a harmonic embedding into a higher-dimensional Euclidean space. The following lemma applies to any (smooth) background metric.

\begin{lemma}\label{lemma5}
Let $g_0$ be a  Riemannian metric on $M$. There exist an open neighborhood $U\subset\R^n$ of $M$, a smooth Riemannian extension of $g_0$ to $U$, an integer $N\ge n$, and real-valued functions $F^1,\ldots,F^N\in C^\infty(U)$ such that
\[
\Delta_{g_0}F^\alpha=0\quad\text{in }U,\qquad \alpha=1,\ldots,N,
\]
and $F=(F^1,\ldots,F^N)\colon U\to\R^N$ is a smooth embedding, hence a diffeomorphism onto its image.
\end{lemma}

\begin{proof} 
Extend $g_0$ smoothly to $\mathbb R^n$, using the extension already fixed above in the conformal case, and choose a compact set $K$ with $M\subset K^{\rm int}$.  By \cite[Section~2, Lemma~4]{GreeneWu75}, finitely many real-valued $g_0$-harmonic functions in a neighborhood of $K$ define a map $F$ which is injective on $K$ and satisfies $\operatorname{rank}DF=n$ there. Since $K$ is compact, $F|_K$ is a homeomorphism onto its image. Restricting to $U= K^{\rm int}$ proves the lemma.
\end{proof}

Note that one can take $N=2n+1$ by the main result in \cite{GreeneWu75}. The embedding $F|_M$ can also be viewed as a finite-dimensional projection of the Poisson embedding in \cite{LLS-Poisson20}, obtained by selecting finitely many boundary data.

\subsection{The harmonic gauge}\label{sec3b}
The next lemma extends Lemma~\ref{lemma1}. In the Euclidean case, the harmonic extension of the identity is the change of coordinates itself. The $g$-harmonic extension of the embedding $F$ from Lemma~\ref{lemma5} may not take values in $F(U)$, so we first project it onto $F(U)$ before applying $F^{-1}$. After changing coordinates, the  difference is normal to $F(U)$; this is the field $Z$ below. For $F=\Id$ and $N=n$, the projection is the identity and $Z=0$.

\begin{lemma}\label{lemma6}
Let $F$ be the harmonic embedding of Lemma~\ref{lemma5} for $g_0$. For $g$ sufficiently close to $g_0$ in $H^s(M)$, there is a diffeomorphism $\Phi\colon M\to M$ fixing $\partial M$ pointwise such that, writing
\begin{equation}\label{20}
\overline{g}=\Phi_*g,\qquad B=a^{-2}(A_{\overline{g}}-A_0),
\end{equation}
there exists $Z\in C^\infty(M;\R^N)$ satisfying
\begin{equation}\label{21}
L_{A_{\overline{g}}}(F+Z)=0,\qquad Z|_{\partial M}=0, \qquad 
\partial_kF\cdot Z=0,\qquad k=1,\ldots,n.
\end{equation}
Moreover,
\begin{equation}\label{22}
\begin{split}
\|\Phi-\Id\|_{H^{s+1}(M)}+\|B\|_{H^s(M)}+\|Z\|_{H^{s+1}(M)}&\le C\|g-g_0\|_{H^s(M)},\\
\|Z\|_{H^1(M)}&\le C\|B\|_{L^2(M)}.
\end{split}
\end{equation}
Here, $C$ depends on $M,g_0,s$ and on the fixed embedding $F$.
\end{lemma}

\begin{proof} 
All norms in this proof are over $M$. 
For $\alpha=1,\ldots,N$, let $F_g^\alpha$ solve
\[
\Delta_gF_g^\alpha=0\quad\text{in }M,\qquad F_g^\alpha|_{\partial M}=F^\alpha|_{\partial M}.
\]
Elliptic regularity up to the boundary implies that $F_g\in C^\infty(M;\R^N)$. Since $L_{A_0}F=0$, the difference $F_g-F$ satisfies
\[
L_{A_0}(F_g-F)=-\diver\bigl((A_g-A_0)\nabla F_g\bigr),\qquad (F_g-F)|_{\partial M}=0.
\]
Standard elliptic Dirichlet estimates and the algebra property of $H^s(M)$ imply
\[
\begin{split}
\|F_g-F\|_{H^{s+1}}&\le C\|(A_g-A_0)\nabla F_g\|_{H^s}\\
&\le C\|A_g-A_0\|_{H^s}\bigl(1+\|F_g-F\|_{H^{s+1}}\bigr).
\end{split}
\]
For $g$ sufficiently close to $g_0$, we can absorb the last term and obtain
\begin{equation}\label{23}
\|F_g-F\|_{H^{s+1}}\le C\|A_g-A_0\|_{H^s}\le C\|g-g_0\|_{H^s}.
\end{equation}
By Sobolev embedding, 
\[
\|F_g-F\|_{C^0(M)}\le C\|g-g_0\|_{H^s}.
\]

Choose a neighborhood $\mathcal T$ of $F(M)$ in $\R^N$ on which the Euclidean nearest-point projection $\pi\colon\mathcal T\to F(U)$ is smooth. Compactness of $F(M)$ and the preceding estimate ensure that $F_g(M)\subset\mathcal T$ for $g$ sufficiently close to $g_0$. We can therefore define
\begin{equation}\label{24}
\Phi=F^{-1}\circ\pi\circ F_g\colon M\to U.
\end{equation}
Since $F_g=F$ on $\partial M$, the map $\Phi$ fixes the boundary pointwise. The smooth composition estimate in \cite[Proposition~2.20]{IKT13}, together with $(F^{-1}\circ\pi)\circ F=\Id$, yields
\begin{equation}\label{25}
\begin{split}
\|\Phi-\Id\|_{H^{s+1}}&=\|(F^{-1}\circ\pi)\circ F_g-(F^{-1}\circ\pi)\circ F\|_{H^{s+1}}\\
&\le C\|F_g-F\|_{H^{s+1}}.
\end{split}
\end{equation}
To apply that estimate, one may extend $F^{-1}\circ\pi$ smoothly to $\R^N$ using a cutoff equal to one on a smaller neighborhood of $F(M)$ containing $F_g(M)$.

By \r{23}, \r{25}, and Sobolev embedding,
\[
\|\Phi-\Id\|_{C^1(M)}\le C\|g-g_0\|_{H^s(M)}.
\]
For $g$ sufficiently close to $g_0$, this bound is less than $1/2$. Since $\Phi$ fixes $\partial M$ pointwise, the argument of Lemma~\ref{lemma1} proves that $\Phi$ is a diffeomorphism of $M$ onto itself, smooth with smooth inverse up to the boundary.

The Sobolev composition and inversion estimates of \cite{IKT13}, applied after extending $\Phi-\Id$ to $\R^n$, imply
\[
\|\Phi^{-1}-\Id\|_{H^{s+1}}\le C\|\Phi-\Id\|_{H^{s+1}}\le C\|g-g_0\|_{H^s}.
\]
To estimate $B$, use the conductivity transformation law
\[
A_{\overline{g}}=\left(\frac{D\Phi\,A_g\,D\Phi^{\top}}{\det D\Phi}\right)\circ\Phi^{-1}.
\]
By \r{20}, this implies
\[
a^2B=\left(\frac{D\Phi\,A_g\,D\Phi^{\top}}{\det D\Phi}-A_0\right)\circ\Phi^{-1}+A_0\circ\Phi^{-1}-A_0.
\]
Since $\det D\Phi\ge2^{-n}$ and $a$ is fixed, smooth, and positive, the Sobolev product and composition estimates in \cite[Lemmas~2.7 and~2.11, Proposition~2.20]{IKT13} yield
\[
\begin{split}
\|B\|_{H^s}&\le C\bigl(\|A_g-A_0\|_{H^s}+\|D\Phi-\Id\|_{H^s}+\|\Phi^{-1}-\Id\|_{H^s}\bigr)\\
&\le C\bigl(\|g-g_0\|_{H^s}+\|\Phi-\Id\|_{H^{s+1}}\bigr)\\
&\le C\|g-g_0\|_{H^s}.
\end{split}
\]

We construct the vector field $Z$ asserted in the lemma. Set 
\begin{equation}\label{26}
Z=F_g\circ\Phi^{-1}-F.
\end{equation}
We verify the required properties below. 
Invariance of the Laplace--Beltrami operator under diffeomorphisms implies $L_{A_{\overline{g}}}(F+Z)=0$. Also, $Z=0$ on $\partial M$ because $F_g=F$ there and $\Phi$ fixes the boundary. By \r{24}, $\pi(F_g(x))=F(\Phi(x))$. Since $\pi$ is the nearest-point projection, we get
\[
F_g(x)-F(\Phi(x))\perp T_{F(\Phi(x))}F(U).
\]
Setting $y=\Phi(x)$ and using \r{26}, we obtain
\[
Z(y)\perp T_{F(y)}F(U)=\operatorname{span}\{\partial_1F(y),\ldots,\partial_nF(y)\},
\]
which proves the normality condition in \r{21}. Moreover,
\[
Z=(F_g-F)\circ\Phi^{-1}+F\circ\Phi^{-1}-F.
\]
The Sobolev composition estimates, \r{23}, and the bound for $\Phi^{-1}-\Id$ above imply
\[
\|Z\|_{H^{s+1}}\le C\bigl(\|F_g-F\|_{H^{s+1}}+\|\Phi^{-1}-\Id\|_{H^{s+1}}\bigr)\le C\|g-g_0\|_{H^s}.
\]

Finally, $L_{A_0}F=0$ and \r{21} imply
\begin{equation}\label{27}
L_{A_{\overline{g}}}Z=-\diver(a^2B\nabla F),\qquad Z|_{\partial M}=0.
\end{equation}
Multiply by $Z$ and integrate by parts to get 
\[
\|\nabla Z\|_{L^2}^2\le C\left|\int_M a^2B^{ij}\partial_jF\cdot\partial_iZ\d x\right|\le C\|B\|_{L^2}\|\nabla Z\|_{L^2}.
\]
The Poincar\'e inequality proves the $H^1$ bound in \r{22}.
\end{proof}

\subsection{Boundary jets}\label{sec3c}

\begin{lemma}\label{lemma7}
The tensor $\overline{g}-g_0$ is flat at $\partial M$.
\end{lemma}

\begin{proof}
As in Lemma~\ref{lemma2}, the boundary determination in \cite[Proposition~1.3]{LeeU} provides a local diffeomorphism $\psi$ in a collar fixing $\partial M$ pointwise, such that $\widehat g=\psi^*g$ and $g_0$ have equal boundary jets. Set $W=F_g\circ\psi-F$. Then
\begin{equation}\label{28}
\begin{gathered}
\Delta_{\widehat g}W^\alpha=(\Delta_{g_0}-\Delta_{\widehat g})F^\alpha,\\
W^\alpha|_{\partial M}=\partial_{\nu_{\widehat g}}W^\alpha|_{\partial M}=0,\qquad \alpha=1,\ldots,N.
\end{gathered}
\end{equation}
The Neumann data in \r{28} follow from DN equality \r{15} and $\widehat g=g_0$ on $\partial M$. The right side is flat; ellipticity makes the boundary noncharacteristic, so the argument in Lemma~\ref{lemma2} shows that $W$ is flat.

The new step in the proof  is to use the projection in \r{24}:
\[
\Phi\circ\psi=F^{-1}\circ\pi\circ(F+W).
\]
Since $F^{-1}\circ\pi$ is smooth near $F(M)$ and $F^{-1}\circ\pi\circ F=\Id$, flatness of $W$ at $\partial M$ implies that $\Phi\circ\psi-\Id$ is flat there as well. Since $(\Phi\circ\psi)_*\widehat g=\overline{g}$, the claim follows.
\end{proof}

Lemma~\ref{lemma7} allows us to extend $\overline{g}$ by $g_0$ outside $M$, and $B$ by zero; then $B\in C_0^\infty(\R^n)$. By \r{22},
\begin{equation}\label{29}
\|B\|_{H^s(\R^n)}\le C\|g-g_0\|_{H^s(M)}.
\end{equation}

The following estimate replaces the divergence-free identity \r{3}.

\begin{lemma}\label{lemma8}
For $B$ as above,
\begin{equation}\label{30}
\|\diver B\|_{L^2(\R^n)}\le C\|B\|_{L^2(\R^n)}.
\end{equation}
\end{lemma}

\begin{proof}
Our aim is to estimate $\diver(a^2B)$ using only the $L^2$ bound for $B$ and the $H^1$ bound for $Z$ from Lemma~\ref{lemma6}. In \r{27} expanded in \r{31} below, the components of $\diver(a^2B)$ multiply the  vectors $\partial_jF$ tangent to $F(U)$. We therefore take inner products of \r{31} with $\partial_kF$. The  identity $\partial_kF\cdot Z=0$ allows us to express these components of $L_{A_{\overline g}}Z$ using only $Z$ and its first derivatives, see \r{33}. The resulting system for $\diver(a^2B)$ has a coefficient matrix close to the invertible Gram matrix of $DF$.

We have in $M$,
\begin{equation}\label{31}
L_{A_{\overline{g}}}Z=-L_{A_{\overline{g}}}F=-\partial_i(a^2B^{ij}\partial_jF)  =-\partial_i(a^2B^{ij})\partial_jF-a^2B^{ij}\partial_i\partial_jF,
\end{equation}
where in the first equality we used Lemma~\ref{lemma6}, and in the second the definition of $B$ and the fact that $F$ is $g_0$-harmonic. Expanding $L_{A_{\overline g}}\partial_kF$ and using $A_{\overline g}=A_0+a^2B$, we obtain
\begin{equation}\label{32}
L_{A_{\overline{g}}}\partial_kF=\bigl(\partial_iA_0^{ij}+\partial_i(a^2B^{ij})\bigr)\partial_j\partial_kF+A_{\overline{g}}^{ij}\partial_i\partial_j\partial_kF.
\end{equation}
Applying $L_{A_{\overline{g}}}$ to $\partial_kF\cdot Z=0$, we obtain
\begin{equation}\label{33}
0=\partial_kF\cdot L_{A_{\overline{g}}}Z+Z\cdot L_{A_{\overline{g}}}\partial_kF+2A_{\overline{g}}^{ij}\partial_i\partial_kF\cdot\partial_jZ.
\end{equation}
Using \r{31}, \r{32}, and \r{33}, and reordering terms, we find that
\[
\begin{split}
\bigl(\partial_kF\cdot\partial_jF-Z\cdot\partial_j\partial_kF\bigr)\partial_i(a^2B^{ij})
&=-a^2B^{ij}\partial_kF\cdot\partial_i\partial_jF\\
&\quad+Z\cdot\bigl(\partial_iA_0^{ij}\partial_j\partial_kF+A_{\overline{g}}^{ij}\partial_i\partial_j\partial_kF\bigr)\\
&\quad+2A_{\overline{g}}^{ij}\partial_i\partial_kF\cdot\partial_jZ.
\end{split}
\]
By Lemma~\ref{lemma5}, the matrix $\partial_kF\cdot\partial_jF$ is uniformly positive definite on $M$. Furthermore, by Lemma~\ref{lemma6} and the embedding $H^{s+1}(M)\hookrightarrow C^0(M)$, we have
\[
\|Z\|_{L^\infty(M)}\le C\|g-g_0\|_{H^s(M)}.
\]
Thus, taking $g$ close enough to $g_0$, we see that $\partial_kF\cdot\partial_jF-Z\cdot\partial_j\partial_kF$ is still uniformly positive definite and therefore uniformly invertible. Since $A_{\overline{g}}$ is uniformly bounded and $a,F$ are fixed and smooth, it follows that
\[
\|\diver(a^2B)\|_{L^2(\R^n)}\le C\bigl(\|B\|_{L^2(\R^n)}+\|Z\|_{H^1(M)}\bigr)\le C\|B\|_{L^2(\R^n)},
\]
where the second inequality follows from Lemma~\ref{lemma6}. Finally,
$\diver B=a^{-2}\diver(a^2B)-2B\nabla\log a $, which proves \r{30}.
\end{proof}

\subsection{The Liouville transform}\label{sec3d}
By \r{20}, the extended conductivity is $A_{\overline{g}}=a^2(\Id+B)$ on $\R^n$, and it equals $A_0$ outside $M$. Set
\begin{equation}\label{34}
b_2=\|B\|_{L^2},\qquad b_s=\|B\|_{H^s},\qquad b_\infty=\|B\|_{L^\infty}.
\end{equation}
All norms below are over $\R^n$ unless another domain is indicated. Since $s-1>n/2$, Sobolev embedding implies
\begin{equation}\label{35}
b_\infty+\|\diver B\|_{L^\infty}\le Cb_s.
\end{equation}
Furthermore, the interpolation used in Section~\ref{sec2} yields
\begin{equation}\label{36}
b_\infty\le C\|B\|_{H^{s-1}}\le Cb_2^{1/s}b_s^{1-1/s}.
\end{equation}
The three quantities in \r{34} have different roles: $b_2$ measures the perturbation in its ``natural'' norm, $b_s$ controls its regularity and high-frequency behavior, and $b_\infty$ will be used to determine the admissible frequency range. As in Proposition~\ref{pr1}, the continuation argument below will require a condition of the form $\tau b_\infty\le c_0$; \r{36} will allow us to choose the frequency cutoff within that range.

Set
\begin{equation}\label{18}
q={\Delta a}/{a}\in C_0^\infty(\R^n),
\end{equation}
where $\Delta$ is the Euclidean Laplacian. 
It is well known, see \r{40}, that the substitution $u=v/a$ removes the background conformal factor from the principal part of the equation $L_{A_{\overline{g}}}u=0$. 
The transformed equation \r{38}, related to $\overline{g}$, contains that background potential $q$ \emph{and} the perturbation terms $\diver(B\nabla v)-q_Bv$.

\begin{lemma}\label{lemma9}
Set
\begin{equation}\label{37}
q_B=a^{-1}\diver(B\nabla a)=\nabla\log a\cdot\diver B+a^{-1}B^{ij}\partial_i\partial_j a.
\end{equation}
Then, for $v\in C^\infty(\R^n)$,
\begin{equation}\label{38}
a^{-1}L_{A_{\overline{g}}}(v/a)=(\Delta-q)v+\diver(B\nabla v)-q_Bv.
\end{equation}
Furthermore,
\begin{equation}\label{39}
\|\diver B\|_{L^2}+\|q_B\|_{L^2} \le Cb_2,\quad \|\diver B\|_{L^\infty}+\|q_B\|_{L^\infty} \le Cb_s.
\end{equation}
The constant $C$ depends only on $M,g_0,s$, the  extension of $a$, and on the harmonic embedding $F$. 
\end{lemma}

\begin{proof}
Identity \r{37} follows by expanding the divergence. Since $A_0=a^2\Id$,
\begin{equation}\label{40}
\frac1a\diver\left(a^2\nabla\left(\frac va\right)\right)=\Delta v-\frac{\Delta a}{a}v.
\end{equation}
Using that $B$ is symmetric, we also have
\begin{equation}\label{41}
\frac1a\diver\left(a^2B\nabla\left(\frac va\right)\right)=\diver(B\nabla v)-\frac va\diver(B\nabla a).
\end{equation}
Adding \r{40} and \r{41} proves \r{38}. The $L^2$ bounds in \r{39} follow from
\[
\|q_B\|_{L^2}\le C\bigl(\|\diver B\|_{L^2}+\|B\|_{L^2}\bigr)
\]
and Lemma~\ref{lemma8}. The same expansion of $q_B$, together with \r{35}, proves the $L^\infty$ bounds.
\end{proof}

Identity \r{38} replaces \r{4} from Section~\ref{sec2}: when $a=1$, we have $q=q_B=0$, and the Euclidean harmonic gauge makes $\diver B=0$. Here the bounds in \r{39} control the additional first-order and zeroth-order terms without losing a derivative in $L^2$.

\subsection{Background C.G.O.\ solutions}\label{sec3e}
We introduce the exact $g_0$-harmonic C.G.O.\ solutions of Sylvester-Uhlmann \cite{SylvesterU87}, which replace the Euclidean exponentials \r{1}. We keep $q$  fixed. 
The remainder estimates are also proved in this form in  \cite[Propositions~3.1 and~3.3]{BalUhlmann10}. 

\begin{lemma}\label{lemma10}
Fix a closed ball $\B$ containing $M$ in its interior and an integer $\ell\ge0$. There are constants $T_*,C>0$ such that, if $\zeta\in\C^n$, $\zeta\cdot\zeta=0$, and $|\zeta|\ge T_*$, then there is a smooth global $g_0$-harmonic function
\begin{equation}\label{42}
u_\zeta^0=a^{-1}e^{x\cdot\zeta}(1+r_\zeta^0)
\end{equation}
satisfying
\begin{equation}\label{43}
\|r_\zeta^0\|_{C^\ell(\B)}\le\frac{C}{|\zeta|}.
\end{equation}
The constants depend on the fixed extended background $g_0$, $\B$, and $\ell$ only.
\end{lemma}

\begin{proof}
We use the Faddeev solution operator $G_\zeta$ for $P_\zeta=\Delta+2\zeta\cdot\nabla$, so that $P_\zeta G_\zeta f=f$; see \cite[Proposition~4.5.4 and Remark~4.5.12]{Gunther-book}. Fix $-1<\sigma<0$ and an integer $j>\ell+n/2$. Let $H_\sigma^j(\R^n)$ have norm
\[
\|f\|_{H_\sigma^j}=\sum_{|\alpha|\le j}\|\langle x\rangle^\sigma\partial^\alpha f\|_{L^2},\qquad \langle x\rangle=(1+|x|^2)^{1/2}.
\]
The known weighted $L^2$ estimate and the commutation of $G_\zeta$ with derivatives imply
\begin{equation}\label{45}
\|G_\zeta f\|_{H_\sigma^j}\le\frac{C}{|\zeta|}\|f\|_{H_{\sigma+1}^j}.
\end{equation}
Since $q\in C_0^\infty(\R^n)$ is fixed, multiplication by $q$ maps $H_\sigma^j$  to $H_{\sigma+1}^j$ continuously, with norm depending on $q,j,\sigma$. Thus \r{45} implies that $f\mapsto G_\zeta(qf)$ has operator norm at most $C/|\zeta|$ on $H_\sigma^j$. Increasing $T_*$ makes this norm at most $1/2$. We seek $r_\zeta^0$ as the solution of
\begin{equation}\label{46}
r_\zeta^0=G_\zeta q+G_\zeta(qr_\zeta^0).
\end{equation}
Equation \r{46} has a Neumann-series solution in $H_\sigma^j(\R^n)$, with
\begin{equation}\label{44}
\|r_\zeta^0\|_{H_\sigma^j}\le2\|G_\zeta q\|_{H_\sigma^j}\le\frac{C}{|\zeta|}\|q\|_{H_{\sigma+1}^j}\le\frac{C}{|\zeta|}. 
\end{equation}
Applying $P_\zeta$ to \r{46} yields $P_\zeta r_\zeta^0=q(1+r_\zeta^0)$, and \r{38} with $B=0$ proves that \r{42} is $g_0$-harmonic. Smoothness follows from elliptic regularity. On the fixed ball $\B$, the weighted norm in \r{44} controls the ordinary $H^j$ norm, so Sobolev embedding proves \r{43}.
\end{proof}

For $\tau\ge1$, use the norm
\begin{equation}\label{47}
X_\tau(v)=\|v\|_{L^2}+\tau^{-1}\|\nabla v\|_{L^2}+\tau^{-2}\|\nabla^2v\|_{L^2},
\end{equation}
equivalent to the norm in \r{9}. The following consequence of Lemma~\ref{lemma4} incorporates the fixed background potential.

\begin{corollary}\label{cor1}
There are constants $T_0\ge1$ and $C>0$, depending only on $\B,n$ and $\|q\|_{L^\infty}$, such that, for $v\in H^2(\R^n)$ supported in $\B$ and $\zeta\in\C^n$ with $\zeta\cdot\zeta=0$ and $|\zeta|=\sqrt2\tau$, we have
\begin{equation}\label{48}
X_\tau(v)\le\frac{C}{\tau}\|(P_\zeta-q)v\|_{L^2},\qquad \tau\ge T_0.
\end{equation}
\end{corollary}

\begin{proof}
For $\tau\ge1$, Lemma~\ref{lemma4} and \r{47} imply
\[
X_\tau(v) \le\frac{C}{\tau}\|P_\zeta v\|_{L^2} \le\frac{C}{\tau}\|(P_\zeta-q)v\|_{L^2}+\frac{C\|q\|_{L^\infty}}{\tau}X_\tau(v).
\]
Choose $T_0$ so that $C\|q\|_{L^\infty}/T_0\le1/2$ and absorb the last term to obtain \r{48}.
\end{proof}

We now construct the $\overline{g}$-harmonic extensions of the boundary values of the background C.G.O.\ solutions of Lemma~\ref{lemma10}, taken with $\ell\ge2$, as in Proposition~\ref{pr1}.

\begin{lemma}\label{lemma11}
Increasing $T_0$ from Corollary~\ref{cor1}, there are constants $c_0,C>0$ such that  the following holds. Let $\zeta\cdot\zeta=0$, $|\zeta|=\sqrt2\tau$, and assume
\begin{equation}\label{49}
b_s\le c_0,\qquad \tau\ge T_0,\qquad \tau b_\infty\le c_0.
\end{equation}
The $\overline{g}$-harmonic solution $u_\zeta$ in $M$ with boundary value $u_\zeta^0$ extends by $u_\zeta^0$ outside $M$ to a smooth global $\overline{g}$-harmonic function of the form
\begin{equation}\label{50}
u_\zeta=a^{-1}e^{x\cdot\zeta}(1+r_\zeta^0+r_\zeta),\qquad r_\zeta\in C_0^\infty(\R^n),\qquad \supp r_\zeta\subset M,
\end{equation}
with
\begin{equation}\label{51}
X_\tau(r_\zeta)\le C\tau b_2.
\end{equation}
The constants $c_0$, $C$ depend only on $M,g_0,s$, on the  extension of $a$,   on the harmonic embedding $F$, and on the ball $\mathcal B$.
\end{lemma}

\begin{proof}
The Dirichlet solution is smooth up to $\partial M$ by elliptic regularity. The extension argument of Proposition~\ref{pr1} applies: \r{15}, invariance of the DN map under $\Phi$, and Lemma~\ref{lemma7} imply that the boundary values and conormal fluxes of $u_\zeta$ and $u_\zeta^0$ agree. Thus the extension is $\overline{g}$-harmonic in the distributional sense across $\partial M$.  
The extended metric $\overline{g}$ is smooth by Lemma~\ref{lemma7}, so interior elliptic regularity implies that the extended solution $u_\zeta$, equal to $u_\zeta^0$ outside $M$, is smooth on $\mathbb R^n$. 
Consequently,
\[
r_\zeta=a e^{-x\cdot\zeta}(u_\zeta-u_\zeta^0)\in C_0^\infty(\R^n),\qquad \supp r_\zeta\subset M,
\]
which proves \r{50}.

Conjugating \r{38} by $e^{x\cdot\zeta}$ and using $(P_\zeta-q)(1+r_\zeta^0)=0$, we obtain
\begin{equation}\label{52}
(P_\zeta-q)r_\zeta =-\bigl[B^{ij}(\partial_i+\zeta_i)(\partial_j+\zeta_j)  +(\partial_iB^{ij})(\partial_j+\zeta_j)-q_B\bigr](1+r_\zeta^0+r_\zeta).
\end{equation}
Choose $T_0$ large enough that Lemma~\ref{lemma10} applies with $\ell\ge2$. Then \r{43} implies
\[
\|1+r_\zeta^0\|_{L^\infty(M)}\le C,\qquad \|\nabla r_\zeta^0\|_{L^\infty(M)}+\|\nabla^2r_\zeta^0\|_{L^\infty(M)}\le C\tau^{-1}.
\]
Together with the $L^2$ bounds in \r{39}, these estimates bound the part of the right-hand side of \r{52} acting on $1+r_\zeta^0$ by $C\tau^2b_2$. For the part acting on $r_\zeta$, the definition \r{47} bounds the terms with coefficient $B$ by $C\tau^2b_\infty X_\tau(r_\zeta)$, while the $L^\infty$ bounds in \r{39} control those with coefficients $\diver B$ and $q_B$ by $C\tau b_sX_\tau(r_\zeta)$. Therefore,
\begin{equation}\label{53}
\|(P_\zeta-q)r_\zeta\|_{L^2}\le C\tau^2b_2+C(\tau^2b_\infty+\tau b_s)X_\tau(r_\zeta).
\end{equation}
Applying Corollary~\ref{cor1} to \r{53} yields
\begin{equation}\label{54}
X_\tau(r_\zeta)\le C\tau b_2+C(\tau b_\infty+b_s)X_\tau(r_\zeta).
\end{equation}
Compared with Proposition~\ref{pr1}, the additional $b_s$ term comes from $\diver B$ and $q_B$; the fixed potential $q$ has already been absorbed in Corollary~\ref{cor1}. Choose $c_0$ so that $2Cc_0\le1/2$. Under \r{49}, the last term in \r{54} can then be absorbed, proving \r{51}.
\end{proof}

\subsection{Estimates on the Fourier transform}\label{sec3g}
We use the polarization construction in \r{eta}, retaining the components in the direction $\omega$ that vanish in Section~\ref{sec2} but are only estimated here by Lemma~\ref{lemma8}. 

\begin{lemma}\label{lemma12}
For every $\omega\in\S^{n-1}$, one can choose $L=L(n)$ pairs $(z_\ell,w_\ell)\in\C^n\times\C^n$ such that
\[
z_\ell\cdot z_\ell=w_\ell\cdot w_\ell=0,\qquad z_\ell+w_\ell=-\i\omega,\qquad |z_\ell|=|w_\ell|=\sqrt2,
\]
and, for every complex symmetric matrix $S$,
\begin{equation}\label{55}
|S|\le C\left(\sum_{\ell=1}^L|S(z_\ell,w_\ell)|+|S\omega|\right).
\end{equation}
Here $S(z,w)=S^{ij}z_i w_j$, $|S|$ is the Frobenius norm, and $C$ depends only on $n$. The vectors may be chosen piecewise smoothly in $\omega$.
\end{lemma}

\begin{proof}
Choose an orthonormal basis of $\omega^\perp$. For each ordered pair $p,q$ of distinct basis vectors, use \r{eta} with $\rho=1$, $a=p$, and $b=q$, and write $z_\pm=\zeta_\pm$, $w_\pm=\eta_\pm$. These vectors satisfy the stated null, sum, and norm conditions. The calculation leading to \r{13}, with the $S(\omega,\omega)$ term retained, reads
\[
S(z_\pm,w_\pm)=-S(p,p)+\frac34 S(q,q)-\frac14 S(\omega,\omega)\pm\i\sqrt3 S(p,q).
\]
Taking the difference controls $S(p,q)$; taking the sum controls $S(p,p)-\frac34 S(q,q)$ up to $C|S\omega|$. Interchanging $p,q$ gives the same invertible two-by-two system as in Section~\ref{sec2}, so both diagonal entries are controlled as well. Since $n\ge3$, every basis vector has an orthogonal counterpart in $\omega^\perp$. Thus all entries of $S$ restricted to $\omega^\perp$ are controlled, while the remaining entries are bounded by $|S\omega|$. This proves \r{55}, uniformly in $\omega$. Smooth local orthonormal frames on a finite cover of the sphere provide the asserted piecewise smooth choice.
\end{proof}

For $S=\widehat B(\xi)$ and $\omega=\xi/|\xi|$, the additional term in \r{55} is controlled by Lemma~\ref{lemma8}. Indeed, $\widehat{\diver B}(\xi)=\i\widehat B(\xi)\xi$, so Plancherel's identity and \r{30} imply
\begin{equation}\label{56}
\left\|\mathbf1_{\{|\xi|>T\}}\widehat B(\xi)\frac{\xi}{|\xi|}\right\|_{L^2}\le\frac{C}{T}\|\diver B\|_{L^2}\le\frac{C}{T}b_2,\qquad T>0.
\end{equation}

We next estimate the errors in the background C.G.O.\ amplitudes.

\begin{lemma}\label{lemma13}
Let $e(x,\xi)=(e^{ij}(x,\xi))$ be measurable in $\xi$ and $C^{n+1}$ in $x$, with support in a fixed compact set in $x$, and suppose that
\begin{equation}\label{57}
|\partial_x^\alpha e(x,\xi)|\le C|\xi|^{-1},\qquad |\alpha|\le n+1,\quad |\xi|\ge T_0.
\end{equation}
For a square-integrable matrix field $B$, define
\begin{equation}\label{58}
E_TB(\xi)=\mathbf1_{\{|\xi|>T\}}\int_{\R^n}e^{-\i x\cdot\xi}e^{ij}(x,\xi)B^{ij}(x)\d x,\qquad T\ge T_0.
\end{equation}
Then
\begin{equation}\label{59}
\|E_TB\|_{L^2}\le\frac{C}{T}\|B\|_{L^2},
\end{equation}
where $C$ depends only on $n$, the fixed compact support, and the constant in \r{57}.
\end{lemma}

\begin{proof}
The formal adjoint of \r{58} satisfies
\[
(2\pi)^{-n}(E_T^*\widehat f)^{ij}=\operatorname{Op}\bigl(\mathbf1_{\{|\xi|>T\}}\overline{e^{ij}}\bigr)f,
\]
where $\operatorname{Op}$ is with the left quantization. We are not claiming that its symbol satisfies the standard symbol estimates. 
The compact support and \r{57} imply
\[
\sup_\xi\sum_{|\alpha|\le n+1}\big\|\partial_x^\alpha\bigl(\mathbf1_{\{|\xi|>T\}}\overline{e^{ij}(\bullet,\xi)}\bigr)\big\|_{L^1}\le\frac{C}{T}.
\]
Applying \cite[Theorem~18.1.11$'$, p.~75]{Hormander3} componentwise, followed by Plancherel's identity and duality, proves \r{59}. The theorem requires no derivatives in $\xi$, so the sharp cutoff is allowed.
\end{proof}

We combine the preceding estimates with the frequency splitting used in Section~\ref{sec2}.

\begin{lemma}\label{lemma14}
After increasing $T_0$ and decreasing $c_0$ from Lemma~\ref{lemma11}, if necessary, the conductivity perturbation $B$ in \r{20} satisfies
\begin{equation}\label{60}
\big\|\mathbf1_{\{|\xi|>T\}}\widehat B\big\|_{L^2}\le C\bigl(R^{n/2+1}b_2^2+T^{-1}b_2+R^{-s}b_s\bigr)
\end{equation}
whenever
\begin{equation}\label{61}
b_s\le c_0,\qquad R\ge T\ge T_0,\qquad Rb_\infty\le c_0.
\end{equation}
The constant $C$ depends only on the fixed data and is independent of $B,T,R$.
\end{lemma}

\begin{proof}
Take the background solutions of Lemma~\ref{lemma10} with the $C^{n+2}(\B)$ bound in \r{43}, increasing $T_0$ accordingly. First let $T<|\xi|\le R$, write $\xi=\rho\omega$, and choose $(z_\ell,w_\ell)$ as in Lemma~\ref{lemma12}. For each pair set $\zeta=\rho z_\ell$ and $\eta=\rho w_\ell$. Then $\zeta+\eta=-\i\xi$, both vectors are null, and $|\zeta|=|\eta|=\sqrt2\rho$. Conditions \r{61} allow us to apply Lemma~\ref{lemma11} with $\tau=\rho$.

The identity \r{19}, applied to $\overline{g}$ using invariance of the DN map, and \r{20} imply
\begin{equation}\label{62}
\int_M a^2B^{ij} (\partial_i u_\zeta^0)\,(\partial_j u_\eta^0)\, \d x=-\int_M a^2B^{ij} (\partial_i(u_\zeta-u_\zeta^0))\,(\partial_j u_\eta^0)\, \d x.
\end{equation}
On the left, the leading factors $a^{-1}$ in the two background solutions cancel $a^2$, while $\zeta+\eta=-\i\xi$ produces the Fourier exponential. This leads to the Fourier transform of $B$, as in Section~\ref{sec2b} and \cite{SylvesterU87}.

By \r{50},
\[
a e^{-x\cdot\zeta}\nabla(u_\zeta-u_\zeta^0)=\nabla r_\zeta+(\zeta-\nabla\log a)r_\zeta.
\]
Thus \r{51} and the background estimate \r{43} yield
\begin{equation}\label{63}
\big\|a e^{-x\cdot\zeta}\nabla(u_\zeta-u_\zeta^0)\big\|_{L^2(M)} \le C\rho^2b_2,\quad 
\big\|a e^{-x\cdot\eta}\nabla u_\eta^0\big\|_{L^\infty(M)} \le C\rho.
\end{equation}
In \r{62}, the two gradient factors contribute $a^{-2}e^{x\cdot(\zeta+\eta)}=a^{-2}e^{-\i x\cdot\xi}$, so $a^2$ cancels and the remaining exponential has modulus one. Cauchy--Schwarz and \r{63} therefore imply
\begin{equation}\label{64}
\rho^{-2}\left|\int_M a^2B^{ij} (\partial_i u_\zeta^0)\, (\partial_j u_\eta^0)\, \d x\right|\le C\rho b_2^2.
\end{equation}
This is the equivalent of \r{12a} in the previous section. 

For the background amplitudes, \r{42}--\r{43} imply
\begin{equation}\label{65}
\frac{a e^{-x\cdot\zeta}}{\rho}\nabla u_\zeta^0
 =\frac{\zeta}{\rho}(1+r_\zeta^0)+\frac{\nabla r_\zeta^0-(1+r_\zeta^0)\nabla\log a}{\rho} =z_\ell+O_{C^{n+1}(\B)}(\rho^{-1}),
\end{equation}
and similarly with $\eta,w_\ell$. These expansions hold for every $\rho\ge T_0$. Fix $\chi\in C_0^\infty(\B^{\rm int})$ equal to one near $M$ and set
\[
e_\ell^{ij}(x,\xi)=\chi(x)\left[\rho^{-2}a^2e^{-x\cdot(\zeta+\eta)} (\partial_i u_\zeta^0)\, (\partial_j u_\eta^0)-z_{\ell i}w_{\ell j}\right],
\]
extending it by zero for $|\xi|<T_0$. By \r{65}, this amplitude satisfies \r{57}. It is measurable in $\xi$ by the phase choice and the Neumann-series construction in Lemma~\ref{lemma10}. Let $E_{\ell,T}$ be the operator \r{58} with amplitude $e_\ell^{ij}$. Since $\supp B\subset M$, for $T<|\xi|\le R$ we have
\begin{equation}\label{66}
\rho^{-2}\int_M a^2B^{ij} (\partial_i u_\zeta^0)\,(\partial_j u_\eta^0)\, \d x=\widehat B(\xi)(z_\ell,w_\ell)+E_{\ell,T}B(\xi).
\end{equation}
Combining \r{64}, \r{66}, and Lemma~\ref{lemma13}, we obtain
\begin{equation}\label{67}
\begin{split}
\big\|\mathbf1_{\{T<|\xi|\le R\}}\widehat B(\xi)(z_\ell,w_\ell)\big\|_{L^2}
&\le Cb_2^2\left(\int_{T<|\xi|\le R}|\xi|^2\d\xi\right)^{1/2}+CT^{-1}b_2\\
&\le C\bigl(R^{n/2+1}b_2^2+T^{-1}b_2\bigr).
\end{split}
\end{equation}
Lemma~\ref{lemma12} and \r{56} now control all components:
\begin{equation}\label{68}
\big\|\mathbf1_{\{T<|\xi|\le R\}}\widehat B\big\|_{L^2}\le C\bigl(R^{n/2+1}b_2^2+T^{-1}b_2\bigr).
\end{equation}
Finally, the Sobolev bound, as in Section~\ref{sec2}, implies
\[
\big\|\mathbf1_{\{|\xi|>R\}}\widehat B\big\|_{L^2}\le CR^{-s}\|B\|_{H^s}=CR^{-s}b_s.
\]
Adding this to \r{68} proves \r{60}.
\end{proof}

\subsection{Compactness}\label{sec3h} 
If the local rigidity statement of Theorem~\ref{thm2} fails, we have a sequence of nonzero perturbations $B_j$ with $\|B_j\|_{H^s}\to0$. We now prove that their $L^2$-normalizations have a  convergent subsequence.

\begin{lemma}\label{lemma15}
Let $B_j\ne0$ be conductivity perturbations satisfying the hypotheses of Lemma~\ref{lemma14}, with constants $C,c_0,T_0$ independent of $j$, and write $b_{2,j},b_{s,j},b_{\infty,j}$ for their norms as in \r{34}. If $b_{s,j}\to0$, then a subsequence of $H_j=B_j/b_{2,j}$ converges  in $L^2(\R^n)$ to a tensor $H$ satisfying
\[
\supp H\subset M,\qquad \|H\|_{L^2}=1.
\]
\end{lemma}

\begin{proof}
Since $b_{2,j}\le b_{s,j}$, the cutoffs
\begin{equation}\label{69}
R_j=b_{s,j}^{-1/(n+2)}\left(\frac{b_{s,j}}{b_{2,j}}\right)^{1/s}
\end{equation}
tend to infinity. By \r{36},
\[
R_jb_{\infty,j}\le Cb_{s,j}^{\,1-1/(n+2)}\longrightarrow0.
\]
Moreover, \r{69} and $s>n/2+1$ imply
\begin{equation}\label{70}
\begin{split}
R_j^{n/2+1}b_{2,j} =b_{s,j}^{1/2}\left(\frac{b_{2,j}}{b_{s,j}}\right)^{1-(n/2+1)/s}\le b_{s,j}^{1/2},\qquad  R_j^{-s}\frac{b_{s,j}}{b_{2,j}} =b_{s,j}^{s/(n+2)}.
\end{split}
\end{equation}
Fix $T\ge T_0$. For all sufficiently large $j$, the conditions \r{61} hold with $R=R_j$. Dividing \r{60} by $b_{2,j}$ and using \r{70}, we obtain
\begin{equation}\label{71}
\limsup_{j\to\infty}\big\|\mathbf1_{\{|\xi|>T\}}\widehat H_j\big\|_{L^2}\le\frac{C}{T}.
\end{equation}
Since $\|H_j\|_{L^2}=1$ and $\supp H_j\subset M$, a subsequence converges weakly in $L^2(\R^n)$ to a tensor $H$ supported in $M$. For each fixed $T$, the map $f\mapsto\widehat f|_{\{|\xi|\le T\}}$ from $L^2(M)$ to $L^2(\{|\xi|\le T\})$ has the square-integrable kernel $e^{-\i x\cdot\xi}$, so it is Hilbert--Schmidt and hence compact. Therefore, along this same subsequence,
\[
\big\|\mathbf1_{\{|\xi|\le T\}}(\widehat H_j-\widehat H)\big\|_{L^2}\longrightarrow0.
\]
Weak lower semicontinuity and \r{71} also imply $\|\mathbf1_{\{|\xi|>T\}}\widehat H\|_{L^2}\le C/T$. Splitting the frequencies at $T$ thus yields
\[
\limsup_{j\to\infty}\|\widehat H_j-\widehat H\|_{L^2}\le\frac{2C}{T}.
\]
Letting $T\to\infty$ and applying Plancherel's identity proves strong convergence in $L^2$. In particular, $\|H\|_{L^2}=1$.
\end{proof}

\subsection{The linearized problem}\label{sec3i}
We identify below covariant and contravariant tensors using $\e$ and denote the Euclidean tensor by $\e$ in either case. 
For a smooth positive function $a$, and a vector field $X$, define the normalized infinitesimal gauge by
\begin{equation}\label{72}
\mathcal G_aX=DX+DX^\top-\bigl(\diver X+2X\cdot\nabla\log a\bigr)\e.
\end{equation}
As shown in the proof of Lemma~\ref{lemma16}, $a^2\mathcal G_aX$ is the variation of $A_0=a^2\Id$ under pushforward by the flow of $X$.

\begin{lemma}\label{lemma16}
Let $X\in H_0^1(M;\R^n)$ and let $u,v\in C^\infty(M)$ be $g_0$-harmonic. Then
\begin{equation}\label{73}
\int_M a^2(\mathcal G_aX)^{ij}(\partial_i u)\,(\partial_j v)\,\d x=0,
\end{equation}
and
\begin{equation}\label{74}
L_{A_0}(-X\cdot\nabla u)=-\diver\bigl(a^2\mathcal G_aX\nabla u\bigr)
\end{equation}
in a distributional sense. The same identities hold on $\R^n$ for compactly supported $X\in H^1(\R^n;\R^n)$ and smooth global solutions of $L_{A_0}u=L_{A_0}v=0$.
\end{lemma}

\begin{proof}
First let $X$ be smooth and compactly supported, and let $\varphi_t$ be its flow. The pushed-forward conductivity is
\[
A_t=\left(\frac{a^2D\varphi_tD\varphi_t^\top}{\det D\varphi_t}\right)\circ\varphi_t^{-1},\qquad \left.\frac{\d}{\d t}A_t\right|_{t=0}=a^2\mathcal G_aX.
\]
Differentiating $L_{A_t}(u\circ\varphi_t^{-1})=0$ proves \r{74}. Testing this identity against $v$, and using $L_{A_0}v=0$ and compact support of $X\cdot\nabla u$, proves \r{73}. Both identities extend by $H^1$ density to the stated classes of vector fields.
\end{proof}

We can call $\mathcal G_aX$ with $X\in H_0^1(M;\R^n)$ \emph{potential tensor fields}, by analogy with the X-ray geodesic transform of symmetric two-tensor fields. Then \r{73} is the analog of the fact that the geodesic X-ray transform of two-tensor fields vanishes on potential ones. 


The next two lemmas prepare the proof that the kernel of the linearization consists exactly of potential tensors $H=\mathcal G_aX$, with $X\in H_0^1(M;\R^n)$. Background complex geometric optics solutions yield a condition on $\widehat H$, which the next lemma converts into the decomposition $H=DX+DX^\top+\beta\e$. By \r{72}, this can be rewritten as $H=\mathcal G_aX+\sigma\e$. Lemma~\ref{lemma18} eliminates the scalar remainder $\sigma$, and we then show that $X$ vanishes outside $M$. In the completion of the proof, the limiting gauge condition forces $X=0$, contradicting the unit norm of the limit obtained by compactness.
\begin{lemma}\label{lemma17}
Let $H\in L^2(\R^n)$ be a real symmetric tensor with compact support, and assume that for almost every $\xi\ne0$ the restriction of $\widehat H(\xi)$ to $\xi^\perp$ is a scalar multiple of the Euclidean form. Then there are compactly supported $X\in H^1(\R^n;\R^n)$ and $\beta\in L^2(\R^n)$ such that
\begin{equation}\label{75}
H=DX+DX^\top+\beta\e.
\end{equation}
\end{lemma}

\begin{proof}
For $\xi=\rho\omega\ne0$, set $S=\widehat H(\xi)$ and define
\begin{equation}\label{76}
\widehat\beta(\xi)=\frac{\tr S-S(\omega,\omega)}{n-1},\qquad
\widehat X(\xi)=\frac{S\omega-\frac12\bigl(S(\omega,\omega)+\widehat\beta(\xi)\bigr)\omega}{\i\rho}.
\end{equation}
The hypothesis implies $S=\i(\xi\otimes\widehat X+\widehat X\otimes\xi)+\widehat\beta\e$, proving \r{75}. Moreover, $|\widehat\beta|+|\xi||\widehat X|\le C|\widehat H|$. Since $H$ has compact support, $\widehat H$ is bounded, and $|\xi|^{-2}$ is integrable near zero for $n\ge3$. These bounds imply $X\in H^1$ and $\beta\in L^2$; the reality of $H$ implies that $X,\beta$ are real.

Outside a ball containing $\supp H$, taking the trace of \r{75} yields
\[
\partial_iX_j+\partial_jX_i=\frac2n(\diver X)\delta_{ij}.
\]
Differentiating and combining the three permuted identities gives, distributionally,
\begin{equation}\label{77}
\partial_i\partial_jX_k=\frac1n\bigl(\delta_{jk}\partial_i\diver X+\delta_{ik}\partial_j\diver X-\delta_{ij}\partial_k\diver X\bigr).
\end{equation}
Taking the divergence in $k$ gives $(n-2)\partial_i\partial_j\diver X=-\delta_{ij}\Delta\diver X$. Tracing this identity shows that $\Delta\diver X=0$, hence $D^2\diver X=0$. By \r{77}, $D^3X=0$ there. Thus $X$ is a polynomial of degree at most two on the connected exterior of the ball; since $X\in L^2$, it vanishes there. Equation \r{75} then shows that $\beta$ vanishes there as well.
\end{proof}

The following lemma is basically the Sylvester--Uhlmann \cite{SylvesterU87} argument but with $\sigma\in L^2$ only. 
\begin{lemma}\label{lemma18}
Let $a>0$ be smooth on $\R^n$ and equal to one outside a ball. If $\sigma\in L^2(\R^n)$ has compact support and
\begin{equation}\label{78}
\int_{\R^n}a^2\sigma\nabla u\cdot\nabla v\, \d x=0
\end{equation}
for every pair of smooth global solutions of $L_{A_0}u=L_{A_0}v=0$, then $\sigma=0$.
\end{lemma}

\begin{proof}
The product identity $L_{A_0}(uv)=2a^2\nabla u\cdot\nabla v$ and \r{78} imply
\begin{equation}\label{79}
\big\langle a^{-2}L_{A_0}\sigma,(au)(av)\big\rangle=0.
\end{equation}
The distribution in this pairing belongs to $H^{-2}$ and has compact support. Fix $\xi\in\R^n$, choose orthonormal $p,q\in\xi^\perp$, and, for $t>|\xi|/2$, set
\begin{equation}\label{80}
\begin{split}
\zeta_t&=tp+\i\left(-\frac\xi2+\sqrt{t^2-\frac{|\xi|^2}{4}}\,q\right),\\
\eta_t&=-tp+\i\left(-\frac\xi2-\sqrt{t^2-\frac{|\xi|^2}{4}}\,q\right).
\end{split}
\end{equation}
These vectors are null, have norm $\sqrt2t$, and satisfy $\zeta_t+\eta_t=-\i\xi$. Apply the background construction of Lemma~\ref{lemma10} with two derivatives on a ball containing $\supp\sigma$. Its solutions satisfy
\[
(au_{\zeta_t}^0)(au_{\eta_t}^0)=e^{-\i x\cdot\xi}(1+r_{\zeta_t}^0)(1+r_{\eta_t}^0)\longrightarrow e^{-\i x\cdot\xi}, \quad \text{as $t\to \infty$} 
\]
in $C^2$ near $\supp\sigma$. Passing to the limit in \r{79} shows that the Fourier transform of $a^{-2}L_{A_0}\sigma$ vanishes. Hence $L_{A_0}\sigma=0$. Elliptic regularity makes $\sigma$ smooth, and testing against $\overline{\sigma}$ gives $\int a^2|\nabla\sigma|^2\d x=0$. Compact support then implies $\sigma=0$.
\end{proof}

The next proposition identifies the kernel of the linearized conductivity DN map at $A_0$ with the potential tensors. For smooth $H$, the bilinear form in \r{81} represents the linearized data in the direction $a^2H$; its $L^2$ formulation is needed for the limit in Lemma~\ref{lemma15}. For the Euclidean background, the kernel characterization below follows from \cite[Theorem~3]{CarsteaF23}, restricted to symmetric tensors. We we extend it to a fixed smooth conformally Euclidean background.

\begin{proposition}\label{pr2}
A real symmetric tensor $H\in L^2(M)$ satisfies
\begin{equation}\label{81}
\int_M a^2H^{ij}(\partial_i u)\,(\partial_j v)\,\d x=0
\end{equation}
for all $g_0$-harmonic $u,v\in C^\infty(M)$ if and only if
\begin{equation}\label{82}
H=\mathcal G_aX\qquad\text{for some }X\in H_0^1(M;\R^n).
\end{equation}
\end{proposition}

\begin{proof}
The reverse implication is Lemma~\ref{lemma16}. For the forward implication, extend $H$ by zero to $\R^n$. Fix $\xi\ne0$ and orthonormal $p,q\in\xi^\perp$, and use the background C.G.O.\ solutions with phases \r{80} in \r{81}. By \r{43}, uniformly on $M$,
\[
t^{-1}ae^{-x\cdot\zeta_t}\nabla u_{\zeta_t}^0\longrightarrow p+\i q,\qquad
t^{-1}ae^{-x\cdot\eta_t}\nabla u_{\eta_t}^0\longrightarrow-p-\i q,
\]
as $t\to\infty$. Dividing \r{81} by $t^2$ and letting $t\to\infty$ therefore yields
\begin{equation}\label{83}
\widehat H(\xi)(p+\i q,p+\i q)=0.
\end{equation}
Replacing $q$ by $-q$ in \r{83} shows that $\widehat H(\xi)(p,p)=\widehat H(\xi)(q,q)$ and $\widehat H(\xi)(p,q)=0$. Thus the restriction to $\xi^\perp$ is scalar. Lemma~\ref{lemma17} provides the compactly supported $X,\beta$ in \r{75}, and
\begin{equation}\label{84}
H-\mathcal G_aX=\sigma\e,\qquad \sigma=\beta+\diver X+2X\cdot\nabla\log a.
\end{equation}
The scalar $\sigma$ belongs to $L^2$ and has compact support. Subtracting the global identity \r{73} from \r{81} shows that \r{78} holds, so Lemma~\ref{lemma18} implies $\sigma=0$.

It remains to prove that $X$ vanishes outside $M$, including any bounded complementary components. Let $\varphi\in C_0^\infty(\R^n\setminus M)$ be real and set $\tilde a=ae^\varphi$. Since $\tilde a=a$ on $M$, the global solutions for $\tilde a^2\Id$ restrict to $g_0$-harmonic functions in $M$, so \r{81} also holds for this background. The decomposition \r{76} depends only on $H$ and therefore produces the same $X,\beta$. Repeating the scalar argument in \r{84} with $\tilde a$ gives $H=\mathcal G_{\tilde a}X$, and hence
\begin{equation}\label{85}
0=\mathcal G_{\tilde a}X-\mathcal G_aX=-2(X\cdot\nabla\varphi)\e.
\end{equation}
On any ball compactly contained in $\R^n\setminus M$, choose $\varphi$ equal to a coordinate function near that ball. Equation \r{85} then forces each component of $X$ to vanish there. Thus the global $H^1$ vector field $X$ is zero outside $M$, and the trace characterization of $H_0^1$ on smooth domains proves \r{82}.
\end{proof}

\subsection{Completion of the proof}\label{sec3j}
\begin{proof}[Proof of Theorem~\ref{thm2}]
Fix $g_0$ and suppose the conclusion fails. Then there are  metrics $g_j$ with $\Lambda_{g_j}=\Lambda_{g_0}$ and $\|g_j-g_0\|_{H^s(M)}\to0$ which are not related to $g_0$ by boundary-fixing diffeomorphisms. Apply Lemmas~\ref{lemma6}--\ref{lemma7} and write $\overline{g}_j=(\Phi_j)_*g_j$, with perturbations $B_j$ as in \r{20}. By \r{29}, $b_{s,j}\to0$. Moreover, $B_j\ne0$: otherwise $A_{\overline{g}_j}=A_0$, hence $\overline{g}_j=g_0$ since $n\ge3$, contradicting the choice of $g_j$. Lemma~\ref{lemma15} provides, after passing to a subsequence,
\[
H_j=B_j/b_{2,j}\longrightarrow H\quad\text{in }L^2,\qquad \|H\|_{L^2}=1.
\]
Set $A_j=A_{\overline{g}_j}=a^2(\Id+B_j)$.By \r{35}, $A_j\to A_0$ in $L^\infty(M)$, and these conductivities are uniformly elliptic for large $j$.

Fix $g_0$-harmonic $u,v\in C^\infty(M)$ and let $u_j$ solve $L_{A_j}u_j=0$ with boundary value $u$. Then
\[
L_{A_j}(u_j-u)=-\diver(a^2B_j\nabla u),\qquad (u_j-u)|_{\partial M}=0.
\]
The energy estimate and Poincar\'e inequality imply
\begin{equation}\label{86}
\|u_j-u\|_{H^1(M)}\le Cb_{2,j},
\end{equation}
where $C$ may depend on the fixed function $u$. Applying \r{19} to $\overline{g}_j$ and using \r{86}, we obtain
\begin{equation}\label{87}
\left|\int_M a^2H_j^{ij}(\partial_i u)\,(\partial_j v)\,\d x\right|
=\frac1{b_{2,j}}\left|\int_M a^2B_j^{ij}(\partial_i(u_j-u))\,(\partial_j v)\,\d x\right|\le Cb_{2,j}\longrightarrow0.
\end{equation}
Strong $L^2$ convergence in \r{87} proves \r{81} for $H$. Proposition~\ref{pr2} therefore gives $H=\mathcal G_aX$ for some $X\in H_0^1(M;\R^n)$.

Let $Z_j$ be the normal fields from Lemma~\ref{lemma6} and set $V_j=Z_j/b_{2,j}$. By \r{22}, this sequence is bounded in $H_0^1(M;\R^N)$, and \r{21}, \r{27} imply
\begin{equation}\label{88}
L_{A_j}V_j=-\diver(a^2H_j\nabla F),\qquad DF^\top V_j=0.
\end{equation}
After passing to a further subsequence, $V_j\rightharpoonup V$ in $H_0^1$. Since $A_j\to A_0$ uniformly and $H_j\to H$ in $L^2$, passing to the weak formulation of \r{88} yields
\begin{equation}\label{89}
L_{A_0}V=-\diver(a^2H\nabla F),\qquad DF^\top V=0.
\end{equation}
Each component of $F$ is $g_0$-harmonic. Thus \r{74}, applied to $H=\mathcal G_aX$, shows that $-DFX\in H_0^1(M;\R^N)$ solves the same Dirichlet problem as $V$ in \r{89}. Uniqueness gives $V=-DFX$. The normality in \r{89} now implies $DF^\top DF\,X=0$. Since $DF$ has full rank by Lemma~\ref{lemma5}, we have $X=0$, hence $H=0$, contradicting $\|H\|_{L^2}=1$.
\end{proof}

\section{Proof of Corollary~\ref{cor_q}} 
We prove the sufficiency  of \r{c1} first, which is known. 
Liouville's transform \r{40} for a general metric $g$ and $a>0$ has the form
\begin{equation}   \label{L}
a^{\frac{n+2}{n-2}} \Delta_{a^{4/(n-2)}g}(a^{-1}v)=\Delta_g v-\frac{\Delta_g a}{a}v.
\end{equation}
By the invariance under boundary-fixing diffeomorphisms, it suffices to consider $\Phi=\Id$. Applying the preceding identity with $g_0$ in place of $g$,
and using \r{c1}, gives 
\[
(-\Delta_g+q)(a^{-1}v)
=a^{-\frac{n+2}{n-2}}(-\Delta_{g_0}+q_0)v.
\]
Thus the two Dirichlet problems are identical up to a multiplication by $a^{-1}$. That factor preserves their boundary values, since $a=1$ on $\bo$. Moreover, $g=g_0$ on $\bo$, so their unit normals agree there.
Since $\d a=0$ on $\bo$, we also obtain $\partial_{\nu_g}(a^{-1}v)|_\bo
=\partial_{\nu_{g_0}}v|_\bo$. Therefore $\Lambda_{g,q}=\Lambda_{g_0,q_0}$, proving sufficiency.

For necessity, set $P=-\Delta_g+q$ and $P_0=-\Delta_{g_0}+q_0$.
For sufficiently small $\varepsilon_0$, the first Dirichlet
eigenvalue of $P$ remains positive. Let $w,w_0$ solve
\[
Pw=0,\qquad P_0w_0=0,\qquad w|_\bo=w_0|_\bo=1.
\]
The maximum principle gives $w,w_0>0$ on $M$. 

Let $\eps$ denote the left-hand side of \r{c0}. We claim that $\|w-w_0\|_{H^{s+1}}\le C\eps$. Indeed, for $z=w-w_0$, we have $z|_\bo=0$ and
$P_0z=(P_0-P)(w_0+z)$. The standard elliptic estimate for the fixed
operator $P_0$ and the Sobolev product estimates give
\[
\|z\|_{H^{s+1}}
\le C\|(P_0-P)(w_0+z)\|_{H^{s-1}}
\le C\eps\bigl(1+\|z\|_{H^{s+1}}\bigr).
\]
Absorbing the last term for small $\eps$ proves the claim. 

Define
\begin{equation}   \label{L2}
\widehat g=w^{4/(n-2)}g,\qquad \widehat g_0=w_0^{4/(n-2)}g_0.
\end{equation}
Then $\widehat g_0$ is a fixed smooth conformally Euclidean metric,
and $\|\widehat g-\widehat g_0\|_{H^s}\le C\eps$. Applying \r{L} with $a$ replaced by $w$ and $v$ replaced by $wv$, and using $\Delta_gw=qw$, gives
\[
(-\Delta_g+q)(wv)
=w^{\frac{n+2}{n-2}}(-\Delta_{\widehat g}v).
\]
Since $w=1$ on $\bo$, the unit normals for $\widehat g$ and $g$
agree there. Thus
\[
\Lambda_{\widehat g}f
=\Lambda_{g,q}f-f\Lambda_{g,q}1.
\]
The analogous identity for the background $g_0$ and the assumed equality of the DN maps (applied to $f$ and then to $1$) imply $\Lambda_{\widehat g}=\Lambda_{\widehat g_0}$.
For sufficiently small $\varepsilon_0$, Theorem~\ref{thm2}
therefore gives
\begin{equation}   \label{L3}
\widehat g=\Phi^*\widehat g_0
\end{equation}
for a  diffeomorphism $\Phi$ fixing $\bo$ pointwise.

Set
\[
h=\Phi^*g_0,\qquad b=\Phi^* w_0,\qquad a=b/w.
\]
Then $a>0$, $a|_\bo=1$, and the above identities, combined with \r{L2}, \r{L3} yield
$g=a^{4/(n-2)}h$. Moreover,
$(-\Delta_h+\Phi^*q_0)b=0$. Applying \r{L} with $g$ replaced by $h$ and $v$ replaced by $b$, and using $w=a^{-1}b$ and $Pw=0$, gives
\[
\left(\frac{\Delta_h a}{a}
+a^{\frac4{n-2}}q-\Phi^*q_0\right)b=0.
\]
Since $b>0$, this proves the formula for $q$ in \r{c1}.

Finally, $g=h$ on $\bo$, so their unit normals agree there.
Since $b=w=1$ on $\bo$, diffeomorphism invariance and the DN equality
give
\[
\partial_{\nu_h}a
=\partial_{\nu_h}b-\partial_{\nu_g}w
=\Lambda_{g_0,q_0}1-\Lambda_{g,q}1=0
\quad\text{on }\bo.
\]
Together with $a|_\bo=1$, this implies $\d a|_\bo=0$ and
completes the proof.

\section*{Statements and Declarations}

\subsection*{Funding}
Sebasti\'an Mu\~noz-Thon was supported by the European Research Council (ERC) under the European Union's Horizon 2020 research and innovation programme (Grant agreement no. 101162990 -- ADG). Plamen Stefanov was partly supported by NSF Grant DMS-2452757.

%


\begin{thebibliography}{10}

\bibitem{Alessandrini88}
G.~Alessandrini.
\newblock Stable determination of conductivity by boundary measurements.
\newblock {\em Appl. Anal.}, 27(1-3):153--172, 1988.

\bibitem{AlessandriniCabib07}
G.~Alessandrini and E.~Cabib.
\newblock Determining the anisotropic traction state in a membrane by boundary
  measurements.
\newblock {\em Inverse Probl. Imaging}, 1(3):437--442, 2007.

\bibitem{AstalaLP05}
K.~Astala, M.~Lassas, and L.~P\"aiv\"arinta.
\newblock {C}alder\'on's inverse problem for anisotropic conductivity in the
  plane.
\newblock {\em Comm. Partial Differential Equations}, 30:207--224, 2005.

\bibitem{BalUhlmann10}
G.~Bal and G.~Uhlmann.
\newblock Inverse diffusion theory of photoacoustics.
\newblock {\em Inverse Problems}, 26(8):085010, 2010.

\bibitem{Burago-Ivanov}
D.~Burago and S.~Ivanov.
\newblock Boundary rigidity and filling volume minimality of metrics close to a
  flat one.
\newblock {\em Ann. of Math. (2)}, 171(2):1183--1211, 2010.

\bibitem{Calderon80}
A.~P. Calder{\'o}n.
\newblock On an inverse boundary value problem.
\newblock In {\em Seminar on numerical analysis and its applications to
  continuum physics (Rio de Janeiro, 1980)}, number~12 in Cole\c{c}\~{a}o Atas,
  pages 65--73. Sociedade Brasileira de Matem\'atica, Rio de Janeiro, 1980.

\bibitem{CarsteaF23}
C.~I. C\^arstea and A.~Feizmohammadi.
\newblock A density property for tensor products of gradients of harmonic functions and applications.
\newblock {\em J. Funct. Anal.}, 284(2):109740, 2023.

\bibitem{CarsteaLT24}
C.~I. C\^arstea, T.~Liimatainen, and L.~Tzou.
\newblock The {C}alder\'on problem on {R}iemannian surfaces and of minimal surfaces.
\newblock Preprint, arXiv:2406.16944, 2024.

\bibitem{DDO13}
M.~Desbrun, R.~D. Donaldson, and H.~Owhadi.
\newblock Modeling across scales: discrete geometric structures in
  homogenization and inverse homogenization.
\newblock In M.~Z. Pesenson, editor, {\em Multiscale Analysis and Nonlinear
  Dynamics}, pages 19--64. Wiley-VCH, 2013.

\bibitem{DosSantosKSU09}
D.~Dos Santos~Ferreira, C.~E. Kenig, M.~Salo, and G.~Uhlmann.
\newblock Limiting {C}arleman weights and anisotropic inverse problems.
\newblock {\em Invent. Math.}, 178(1):119--171, 2009.

\bibitem{DosSantosKLS16}
D.~Dos Santos~Ferreira, Y.~Kurylev, M.~Lassas, and M.~Salo.
\newblock The {C}alder\'on problem in transversally anisotropic geometries.
\newblock {\em J. Eur. Math. Soc. (JEMS)}, 18(11):2579--2626, 2016.

\bibitem{Gunther-book}
J.~Feldman, M.~Salo, and G.~Uhlmann.
\newblock {\em The {C}alder\'on problem---an introduction}, volume 253 of {\em
  Graduate Studies in Mathematics}.
\newblock American Mathematical Society, Providence, RI, [2025] \copyright
  2025.

\bibitem{GreeneWu75}
R.~E. Greene and H.~Wu.
\newblock Embedding of open {R}iemannian manifolds by harmonic functions.
\newblock {\em Ann. Inst. Fourier (Grenoble)}, 25(1):215--235, 1975.

\bibitem{Gromov}
M.~Gromov.
\newblock Filling {R}iemannian manifolds.
\newblock {\em J. Differential Geom.}, 18(1):1--147, 1983.

\bibitem{GuillarmouTzou11}
C.~Guillarmou and L.~Tzou.
\newblock {C}alder\'on inverse problem with partial data on {R}iemann surfaces.
\newblock {\em Duke Math. J.}, 158(1):83--120, 2011.

\bibitem{Hormander3}
L.~H{\"o}rmander.
\newblock {\em The analysis of linear partial differential operators. {III}},
  volume 274.
\newblock Springer-Verlag, Berlin, 1985.
\newblock Pseudodifferential operators.

\bibitem{IKT13}
H.~Inci, T.~Kappeler, and P.~Topalov.
\newblock On the regularity of the composition of diffeomorphisms.
\newblock {\em Mem. Amer. Math. Soc.}, 226(1062):vi+60, 2013.

\bibitem{LLS-Poisson20}
M.~Lassas, T.~Liimatainen, and M.~Salo.
\newblock The {P}oisson embedding approach to the {C}alder\'on problem.
\newblock {\em Math. Ann.}, 377(1-2):19--67, 2020.

\bibitem{LassasTU03}
M.~Lassas, M.~Taylor, and G.~Uhlmann.
\newblock The {D}irichlet-to-{N}eumann map for complete {R}iemannian manifolds
  with boundary.
\newblock {\em Comm. Anal. Geom.}, 11(2):207--221, 2003.

\bibitem{LassasU01}
M.~Lassas and G.~Uhlmann.
\newblock On determining a {R}iemannian manifold from the
  {D}irichlet-to-{N}eumann map.
\newblock {\em Ann. Sci. \'{E}cole Norm. Sup. (4)}, 34(5):771--787, 2001.

\bibitem{LeeU}
J.~M. Lee and G.~Uhlmann.
\newblock Determining anisotropic real-analytic conductivities by boundary
  measurements.
\newblock {\em Comm. Pure Appl. Math.}, 42(8):1097--1112, 1989.

\bibitem{Lin26}
Y.-H. Lin.
\newblock The anisotropic {C}alder\'on problem: rigidity near the {E}uclidean
  metric.
\newblock Preprint, arXiv:2609.17261, 2026.

\bibitem{Mandache}
N.~Mandache.
\newblock Exponential instability in an inverse problem for the {S}chr\"odinger
  equation.
\newblock {\em Inverse Problems}, 17(5):1435--1444, 2001.

\bibitem{Nachman_98}
A.~I. Nachman.
\newblock Reconstructions from boundary measurements.
\newblock {\em Ann. of Math. (2)}, 128(3):531--576, 1988.

\bibitem{ORS-Lorentzian26}
L.~Oksanen, Rakesh, and M.~Salo.
\newblock Rigidity in the {L}orentzian {C}alder\'on problem with formally
  determined data.
\newblock {\em Comm. Amer. Math. Soc.}, 6:749--792, 2026.

\bibitem{ORS-Semiglobal26}
L.~Oksanen, Rakesh, and M.~Salo.
\newblock Semiglobal uniqueness for the {L}orentzian {C}alder\'on problem,
  2026.
\newblock arXiv:2608.13116.

\bibitem{Salo08}
M.~Salo.
\newblock {C}alder\'on problem.
\newblock Lecture notes, University of Helsinki, 2008.
\newblock Spring 2008.
  \url{https://users.jyu.fi/~salomi/lecturenotes/calderon_lectures.pdf}.


\bibitem{SU-JFA}
P.~Stefanov and G.~Uhlmann.
\newblock Stability estimates for the hyperbolic {D}irichlet to {N}eumann map
  in anisotropic media.
\newblock {\em J. Funct. Anal.}, 154(2):330--358, 1998.

\bibitem{SU-JFA09}
P.~Stefanov and G.~Uhlmann.
\newblock Linearizing non-linear inverse problems and an application to inverse
  backscattering.
\newblock {\em J. Funct. Anal.}, 256(9):2842--2866, 2009.

\bibitem{Sylvester90}
J.~Sylvester.
\newblock An anisotropic inverse boundary value problem.
\newblock {\em Comm. Pure Appl. Math.}, 43(2):201--232, 1990.

\bibitem{SylvesterU87}
J.~Sylvester and G.~Uhlmann.
\newblock A global uniqueness theorem for an inverse boundary value problem.
\newblock {\em Ann. of Math. (2)}, 125(1):153--169, 1987.

\bibitem{uhl_syl}
J.~Sylvester and G.~Uhlmann.
\newblock Inverse problems in anisotropic media.
\newblock In {\em Inverse scattering and applications ({A}mherst, {MA}, 1990)},
  volume 122 of {\em Contemp. Math.}, pages 105--117. Amer. Math. Soc.,
  Providence, RI, 1991.

\end{thebibliography}

\end{document}